\documentclass[12pt]{article}
\usepackage{amsmath,amssymb}
\usepackage{amsfonts}
\usepackage{mathrsfs}
\usepackage{color}
\usepackage{graphicx}

\usepackage{array,booktabs} 
\usepackage{amsmath,amssymb} 
\usepackage{stfloats} 
\usepackage{psfrag} 
\usepackage{color}

\definecolor{box_color}{rgb}{.8,.8,.8}

\usepackage[utf8]{inputenc}
\usepackage[T1]{fontenc}
\usepackage{lmodern}

\newtheorem{lemma}{Lemma}
\newtheorem{proposition}{Proposition}
\newtheorem{corollary}{Corollary}
\newtheorem{fact}{Fact}

\newtheorem{remark}{Remark}

\newtheorem{assumption}{Assumption}
\newtheorem{ass}{C.}

\newenvironment{proof}[1][Proof]{\begin{trivlist}
\item[\hskip \labelsep {\bfseries #1}]}{\end{trivlist}}

\def\begcen{\begin{center}}
\def\endcen{\end{center}}

\def\mx{\mathcal{X}}

\newcommand{\RE}{\mathbb {R}}    

\newcommand{\col}{ \mbox{col} }
\newcommand{\rank}{ \mbox{rank } }

\def\hal{{1 \over 2}}

\def\L2{{\cal L}_2}
\def\L2e{{\cal L}_{2e}}

\def\rea{\mathbb{R}}
\def\sign{\mbox{sign}}

\def\begequarr{\begin{eqnarray}}
\def\endequarr{\end{eqnarray}}
\def\begequarrs{\begin{eqnarray*}}
\def\endequarrs{\end{eqnarray*}}
\def\begarr{\begin{array}}
\def\endarr{\end{array}}
\def\begequ{\begin{equation}}
\def\endequ{\end{equation}}
\def\lab{\label}
\def\begdes{\begin{description}}
\def\enddes{\end{description}}
\def\begenu{\begin{enumerate}}
\def\begite{\begin{itemize}}
\def\endite{\end{itemize}}
\def\endenu{\end{enumerate}}

\def\lef[{\left[\begin{array}}
\def\rig]{\end{array}\right]}
\def\qed{\hfill$\Box \Box \Box$}
\def\begcen{\begin{center}}
\def\endcen{\end{center}}
\def\begrem{\begin{remark}\rm}
\def\endrem{\end{remark}}
\def\begassum{\begin{assumption}}
\def\endassum{\end{assumption}}
\def\begassums{\begin{assumption*}}
\def\endassums{\end{assumption*}}
\def\begassu{\begin{ass}}
\def\endassu{\end{ass}}
\def\beglem{\begin{lemma}}
\def\endlem{\end{lemma}}
\def\begcor{\begin{corollary}}
\def\endcor{\end{corollary}}
\def\begfac{\begin{fact}}
\def\endfac{\end{fact}}

\def\AUT{{\it Automatica}}

\def\CEP{{\it Control Engineering Practice}}

\usepackage{color}

\begin{document}

\title{Global Stabilisation of Underactuated Mechanical Systems  via PID Passivity-Based Control} 
\author{{Jose Guadalupe Romero\thanks{J. G. Romero  is with  Departamento Acad\'emico  de Sistemas Digitales, ITAM, R\'io Hondo 1, 01080, Ciudad de M\'exico, M\'exico, e-mail: jose.romerovelazquez@itam.mx}} ,  {Alejandro Donaire
\thanks{Alejandro Donaire is with   PRISMA Lab, Universit\`a degli Studi di Napoli Federico II, Via Claudio 21, 80125, Naples, Italy,  e-mail: alejandro.donaire@unina.it}}\;  and   {Romeo Ortega
\thanks{R. Ortega is  with Laboratoire des Signaux et Syst\`emes, CNRS--SUPELEC, Plateau du Moulon, 91192
Gif--sur--Yvette, France, e-mail: ortega@lss.supelec.fr}}}
\date{}
\maketitle
\begin{abstract}
In this note we identify a class of underactuated mechanical systems whose desired constant equilibrium position can be globally stabilised with the ubiquitous PID controller.  The class is characterised via some easily verifiable conditions on the systems inertia matrix and potential energy function, which are satisfied by many benchmark examples. The design proceeds in two main steps, first, the definition of two new passive outputs whose weighted sum defines the signal around which the PID is added.  Second, the observation that it is possible to construct a Lyapunov function for the desired equilibrium via a suitable choice of the aforementioned weights and the PID gains and initial conditions. The results reported here follow the same research line as \cite{DONetal} and  \cite{ROMetal}---bridging the gap between the Hamiltonian and the Lagrangian formulations used, correspondingly, in these papers. Two additional improvements to our previous works are the removal of a non-robust cancellation of a potential energy term and the establishment of equilibrium attractivity under weaker assumptions.
\end{abstract}
%
\section{Introduction}
\lab{sec1}
%
A major breakthrough in robotics was the proof by Takegaki and Arimoto \cite{TAKARI} that, in spite of its highly complicated nonlinear dynamics, a simple PD law provides a global solution to the point-to-point positioning task for fully actuated robot manipulators. As shown in \cite{ORTSPO}  the key property underlying the success of such a simple scheme is the passivity of the system dynamics. Indeed, as is now well-known, mechanical systems define passive maps from the external forces to the generalized coordinate velocities, see \cite{KELORT} for an early reference. Invoking this property the derivative term of the aforementioned PD is assimilated to a constant feedback around the passive output, while the proportional one adds a quadratic term to the systems potential energy to assign a minimum at the desired equilibrium, making the total energy function a {\em bona fide} Lyapunov function. It should be mentioned that  this ``energy--shaping plus damping injection" construction proposed 34 years ago is still the basis of most developments in Passivity-Based Control (PBC)---a term coined in \cite{ORTSPO} to describe a controller design procedure where the control objective is achieved via passivation.

While fully actuated mechanical systems admit an arbitrary shaping of the potential energy by means of feedback, and therefore stabilization to any desired equilibrium via PD control, this is in general not possible for {\em underactuated} systems, that is, for systems where the number of degrees of freedom is strictly larger than the number of control variables. In certain cases this problem can be overcome by also modifying the kinetic energy of the system---as done, for instance in interconnection and damping assignment (IDA) PBC \cite{ORTROMDON}, the controlled Lagrangians \cite{BLOetal} or the canonical tranformation \cite{FUJSUG} methods. There are two major drawbacks to this total energy shaping controllers, first, that they require the use of complicated full-state feedback controllers, which makes them more fragile and difficult to tune. Second, that the derivation of the control law relies on the solution of a partial differential equation (PDE), a difficult task that hampers its wider application.  See \cite{HATbook,ORTROMDON} for recent reviews of the literature of PBC of mechanical systems.  

A key feature of total energy shaping methods is that the mechanical structure of the system is preserved in closed--loop, a condition that gives rise to the aforementioned PDE, which characterise the assignable energy functions. In \cite{DONetal} it was recently proposed to {\em relax} this constraint, and  concentrate our attention on the energy shaping objective only. That is, we look for a control law that stabilizes the desired equilibrium assigning to the closed--loop a Lyapunov function of the same form as the energy function of the open--loop system but with new, desired inertia matrix and potential energy function. However, we {\em do not require} that the closed--loop system is a mechanical system with  this Lyapunov function qualifying as its energy function. In this way, the need to solve the PDE is avoided.

The controller design of \cite{DONetal} is carried out proceeding from a {\em Lagrangian representation} of the system and consists of four steps. First,  the application of a (collocated) partial feedback linearization stage, {\em \`a la}  \cite{SPO}. Second,  following \cite{ACOetal}, the identification of conditions on the inertia matrix and the potential energy function that ensure the Lagrangian structure is preserved. As a corollary of the Lagrangian structure preservation two new passive outputs are easily identified. Third, a PID controller around a suitable combination of these passive outputs is applied. Now, as is well known, PID controllers define input strictly passive mappings. Thus, the passivity theorem allows to immediately conclude output strict passivity---hence, ${\cal L}_2$-stability---of the closed--loop system. To, furthermore, achieve the equilibrium stabilization objective a fourth step is required. Namely, to impose some {\em integrability} assumptions on the systems inertia matrix to ensure that the integral of the passive output, {\em i.e.}, the integrator state, can be expressed as a function of the systems generalised coordinates. Two objectives are achieved in this way, first, to assign to the closed-loop an equilibrium at the desired position. Second, adding this function to the systems storage function, to generate a {\em bona fide} Lyapunov function by assigning a minimum at the aforementioned equilibrium.  It is  shown in \cite{DONetal} that many benchmark examples satisfy the  (easily verifiable) conditions on the systems inertia matrix and its potential energy function imposed by the method.

It is widely recognised  that feedback linearization, which involves the exact cancelation of nonlinear terms, is  intrinsically non-robust. Interestingly, it has recently been shown in  \cite{ROMetal} that, for the class of systems considered in \cite{DONetal}, it is possible to identify two new passive outputs {\em without} the feedback linearization step.  The derivations in  \cite{ROMetal} are done working with the {\em Hamiltonian representation} of the system, and the  main modification is the introduction of a suitable momenta coordinate change that directly reveals the new passive outputs. These passive outputs are {\em different} from the ones obtained in \cite{DONetal} using the Lagrangian formalism. Therefore, it is not possible to compare the realms of applicability of both methods---see Remark 6 in  \cite{ROMetal}. One of the main objectives of this paper is to bridge this gap between these two approaches, this is done translating the derivations of \cite{ROMetal} to the Lagrangian framework. It turns out that this, apparently simple, transposition of results from one representation to another is far from evident and requires to establish some new structural properties of the Lagrangian system that, to the best of the authors' knowledge, have not been reported in the literature.  Moreover, as shown in the paper, the analysis in the Lagrangian framework is more insightful than the  Hamiltonian one   as it reveals some interesting connections between the passive outputs that are obscured in the momenta coordinate change used in \cite{ROMetal}. 

Two additional improvements to the controllers of \cite{DONetal,ROMetal} included in the present paper are as follows.
\begite
\item To avoid the cancellation of (part of) the potential energy function needed in \cite{DONetal,ROMetal}. Instead, we identify a class of potential energy functions for which this non-robust operation is {\em obviated} and passivity established with a {\em new} storage function. The example of cart-pendulum in an inclined plane is used to illustrate this extension.
\item In \cite{DONetal,ROMetal} Lyapunov stability of the desired equilibrium is established assigning a Lyapunov function to the closed-loop system. To ensure {\em positive definiteness} of this function additional conditions are imposed on the PID gains. We show here that these conditions are not required if we abandon the Lyapunov stabilisation objective and---invoking LaSalle's invariance principle---establish {\em equilibrium attractivity} only. Although the latter property is admittedly weaker it ensures satisfactory behavior in many applications.  
\endite

The remainder of the paper is structured as follows. In Section \ref{sec2} the new passive outputs are identified. In Section \ref{sec3} we carry out the ${\cal L}_2$-stability analysis, while in Section \ref{sec4} we show that a PID control can assign the desired equilbirium and shape the energy function to make it a Lyapunov function. Section  \ref{sec5} contains the two extensions to \cite{DONetal,ROMetal} mentioned above. In Section \ref{sec6} we illustrate the results with the examples of linear systems and a cart-pendulum on an inclined plane. We wrap-up the paper with concluding remarks and future research in Section \ref{sec7}. To enhance readability the proofs, which are more technical, are given in the appendices.\\
 
%
\noindent {\bf Notation.} $I_n$ is the $n \times n$ identity matrix and $0_{n \times s}$ is an $n \times s$ matrix of zeros, $0_n$ is an $n$--dimensional column vector of zeros, $e_i\in \rea^n,\; i \in \bar n := \{1,\dots,n\}$, are the $n$-dimensional Euclidean basis vectors. For $x \in \rea^n$, $S \in \rea^{n \times n}$, $S=S^\top>0$, we denote the Euclidean norm $|x|^2:=x^\top x$, and the weighted--norm $\|x\|^2_S:=x^\top S x$. All mappings are assumed smooth. Given a function $f:  \rea^n \to \rea$ we define the differential operator $\nabla f:=\left(\frac{\displaystyle \partial f }{\displaystyle \partial x}\right)^\top$. 
%
\section{New Passive Outputs}
\lab{sec2}
%
Following \cite{DONetal,ROMetal} the first step for the characterisation of a class of mechanical systems whose position can be regulated with classical PIDs is the identification of two new passive outputs. As explained in Section \ref{sec1} the derivations of \cite{DONetal} proceed from a Lagrangian description of the system while those of \cite{ROMetal} rely on its Hamiltonian formulation. Moreover, in  \cite{DONetal} it is assumed that the system is given in Spong's normal form \cite{SPO}, which is obtained doing a preliminary partial feedback linearization to the system. This step is obviated in  \cite{ROMetal} introducing a momenta change of coordinates. These discrepancies lead to the identification of two different classes of systems that are stabilisable via PID. To bridge the gap between the two formulations we give in this section the Lagrangian equivalent of  the passive outputs identified in  \cite{ROMetal}.     
%
\subsection{Identification of the class}
\lab{subsec21}
We consider general underactuated mechanical systems  whose dynamics is described by the well known Euler-Lagrange equations of motion
 \begequarr 
 \label{lagr}
M(q) \ddot{q} + C(q,\dot{q}) \dot{q} + \nabla V(q) = G(q)  \tau,
 \endequarr
where $q \in \rea^n$ are the configuration variables, $\tau \in \RE^m$, with $m < n$, are the control signals, $M:\rea^n \to \rea^{n \times n}$, is the positive definite generalized inertia matrix, $C(q,\dot{q})\dot q$ are the Coriolis and centrifugal forces, with $C: \rea^n \times \rea^n \to \rea^{n \times n}$ defined via the Christoffel symbols of the first kind \cite{KELbook,ORTSPO}, $V: \rea^n \to \rea$ is the systems potential energy and $G: \rea^n \to  \rea^{n \times m}$ is the input matrix, which is assumed to verify the following assumption.
\begite
\item[{\bf A1.}] The input matrix is {\em constant} and  of the form
\begequ
\lab{g}
G=\lef[{c} 0_{{s}\times m} \\ I_m  \rig],
\endequ
\endite
where $s:=n-m$.

To simplify the definition of the class, and in agreement with Assumption {\bf A1}, we  partition the generalised coordinates into its unactuated and actuated components as $q=\col(q_u,q_a)$, with $q_u \in \RE^s$ and $q_a \in \RE^{m}$. Similarly,  the inertia matrix is partitioned as
\begin{equation}
\label{Mblocks}
 M (q)= \left[ \begarr{cc} m_{uu} (q) & m^\top_{au} (q) \\ m_{au} (q) & m_{aa}  (q) \endarr \right],
\end{equation}
where $m_{uu}:\rea^n \to \RE^{s\times s}$, $m_{au}:\rea^n \to \RE^{m \times s}$ and $m_{aa}:\rea^n \to \RE^{m\times m}$. Using this notation the class is identified imposing the following assumptions.

\begite
\item[{\bf A2.}] The inertia matrix depends {\em only} on the unactuated variables $q_u$, {\em i.e.},  $M(q)=M(q_u)$.
\item[{\bf A3.}] The sub-block matrix $m_{aa}$ of the inertia matrix is {\em constant}.
\item[{\bf A4.}] The potential energy can be written as
$$
V(q)=V_a(q_a)+V_u(q_u),
$$
and $V_u(q_u)$ is bounded from below.
\endite

The lower boundedness assumption on $V_u(q_u)$ can be relaxed and is only introduced to simplify the notation avoiding the need to talk about cyclo-passive (instead of passive) outputs---see Footnote 2 of \cite{ROMetal}.  

%
\subsection{Passive outputs and storage functions}
\lab{subsec22}
The proposition below, whose proof is given in Appendix A, identifies two new passive outputs used in the PID controller design.

\begin{proposition} \em
\label{pro1}
Consider the underactuated mechanical system \eqref{lagr} verifying Assumptions {\bf A1}-{\bf A4}. Define the new input
\begequ
\lab{utau}
u = \tau - \nabla V_a(q_a)
\endequ
and the outputs 
\begin{eqnarray}
\nonumber
y_u & := & - m_{aa}^{-1}m_{au}(q_u)\dot q_u \\
y_a & := &  m_{aa}^{-1}m_{au}(q_u)\dot q_u + \dot q_a.
\label{outlag}
\end{eqnarray}
The operators $u \mapsto   y_u$ and  $u \mapsto   y_a$ are {\em passive} with storage functions  
\begequarr
H_u(q_u,\dot {q}_u) & = &  \hal \dot q_u^\top m_{uu}^s(q_u)\dot q_u + V_u(q_u) 
 \lab{newstofun1} \\
H_a(q_u,\dot{q}) & = & \hal  \dot q^\top  M_a(q_u) \dot q,
\lab{newstofun2}
\endequarr
where 
\begequarr
\nonumber
m_{uu}^s(q_u) & := & m_{uu}(q_u)-m_{au}^\top(q_u) m_{aa}^{-1}m_{au}(q_u)\\
\lab{ma}
M_a (q_u) & := & \left[ \begarr{cc} m_{au}^\top(q_u) m_{aa}^{-1}m_{au}(q_u) & m^\top_{au} (q_u) \\ m_{au}(q_u)  & m_{aa}   \endarr \right]. 
\endequarr
More precisely,
\begequarr
\lab{dotha}
\dot{ H }_a & = & u^\top   y_a \\ 
\dot{  H}_u & = &  u^\top  y_u.
\lab{dothu}
\endequarr
\end{proposition}
%
%
\subsection{Discussion}
\lab{subsec23}

Assumption {\bf A1} means that the input force co-distribution is integrable \cite{BULLEW}. The interested reader is referred to \cite{OLF} where some conditions of transformability of a general input matrix to the, so-called, no input-coupling form \eqref{g} are given. Assumption {\bf A2} implies that the shape coordinates coincide with the un-actuated coordinates and are referred as ``Class II" systems in \cite{OLF}. 

Assumptions {\bf A3} and {\bf A4} are technical conditions. In particular, the latter condition is required because, as clear from \eqref{utau}, the first step in our design is to apply a preliminary feedback that cancels the term $V_a(q_a)$. See Subsection \ref{subsec41} for a removal of this condition for afine functions $V_a(q_a)$. The class of systems verifying all these assumptions contains many benchmark examples, including the spider crane, the 4-DOF overhead crane and the spherical pendulum on a puck. 

In  \cite{OLF}, see also \cite{OLFtac}, similar assumptions are imposed to identify mechanical systems that are transformable, via a change of coordinates, to the classical feedback and feedforward forms. The motivation is to invoke the well-known backstepping or forwarding techniques of nonlinear control to design stabilising controllers. It should be noticed that, although these two techniques are claimed to be systematic, similarly to total energy shaping techniques \cite{ORTROMDON} and {\em in contrast} with the methodology proposed here, they still require the solution of PDEs, see Subsections 4.2 and 4.3 of \cite{ASTbook}.
    
From \eqref{outlag}-\eqref{ma} it is clear that $y_u+y_a=\dot q_a$ and $H_u(q_u,\dot q)+H_a(q_u,\dot q)=H(q_u,\dot q)$, with 
$$
H(q_u,\dot q)=\hal \dot q^\top M(q_u) \dot q + V_u(q_u), 
$$ 
that is, the total co-energy of the system in closed-loop with \eqref{utau}. In other words, the new passive outputs are obtained splitting the $(1,1)$ block of the kinetic energy function into two components with one containing the Schur complement of the  $(2,2)$ block, which is assumed constant. Similar interpretations are available for the Hamiltonian derivation of the passive outputs reported in \cite{ROMetal}, but in this case for the system represented in the new coordinates.  

We are now in position to reveal the relationship between the passive outputs \eqref{outlag} and the ones reported in \cite{ ROMetal}, denoted here $Y_u$ and $Y_a$, which is as follows
$$
\lef[{c}  Y_u \\ Y_a \rig]=\lef[{cc} m_{aa} & 0_{m \times m} \\ I_m & I_m \rig]\lef[{c}  y_u \\ y_a \rig].
$$  
See Remark 5 of \cite{ROMetal}.
%
\section{PID Control: Well-posedness and ${\cal L}_2$-Stability}
\lab{sec3}
%
 Similarly to \cite{DONetal,ROMetal} we propose the  PID-PBC 
\begequarr
\nonumber
k_e {u} & = & -\left(K_P  y_d+{K_I }z_1 + K_D \dot{{ y}}_d \right)\\
\dot z_1 & = & y_d,\;z_1(0)=z_1^0,
\lab{pidcon1}
\endequarr
where 
\begequ
\lab{yd}
 y_d:=k_a  y_a + k_u  y_u,
\endequ
and we have to select the {\em nonzero} constants $k_e,k_a,k_u \in \rea,\;k_a \neq k_u$, the matrices $K_P,K_I,K_D \in \rea^{m \times m},\; K_P,K_I>0$, $K_D \geq 0$ and the {\em constant} vector $z_1^0 \in \rea^m$. Notice that, replacing  \eqref{outlag} in \eqref{yd}, we can write $y_d$ in the form
\begequ
\lab{yd1}
y_d = k_a \dot q_a + (k_a - k_u)m_{aa}^{-1}m_{au}(q_u)\dot q_u.
\endequ

We remark the presence of an unusual (sign-indefinite) gain $k_e$ in \eqref{pidcon1} and the fact that the initial conditions of the integrator $z_1$ are {\em fixed}. The motivation for the former is to give more flexibility to achieve Lyapunov stabilization and is discussed in remark {\bf (R2)} in Subsection \ref{subsec32}.  On the other hand, imposing $z_1(0)=z_1^0$ to the control is required to assign the desired equilibrium point to the closed-loop for Lyapunov stabilization as it is thoroughly explained in Section \ref{sec4}. It should be underscored that for the input-output analysis carried out in this section the PID \eqref{pidcon1} can be implemented with arbitrary initial conditions for $z_1$. 
 %
\subsection{Well-posedness condition}
\lab{subsec31}
Before proceeding to analyse the stability of the closed-loop it is necessary to ensure that the control law \eqref{pidcon1} can be computed without differentiation nor singularities that may arise due to the presence of the derivative term $\dot y_d$.  For, after some lengthy but straightforward calculations to compute $\dot y_d$,  we can prove that \eqref{pidcon1} is equivalent to
\begequ
\lab{bark}
 K(q_u) {u} =   -K_P  y_d-{K_I }z_1 -  {S}(q,\dot q),
\endequ
where the mapping $K:\rea^s \to \rea^{m \times m}$ is given by
\begequ
\lab{bark0}
\hspace{-8mm}  K(q_u):= k_e I_m + k_a K_D m_{aa}^{-1} + k_u K_D  m_{aa}^{-1}m_{au}(m^s_{uu})^{-1}m_{au}^\top m_{aa}^{-1}.
\endequ
 and  $ S:\rea^n \times \rea^s \mapsto \rea^m$ is the globally defined mapping
 \begin{eqnarray}
{S}(q,\dot q_u) & := & -k_u K_D \Big[ m_{aa}^{-1}\dot{m}_{au}\dot q_u +m_{aa}^{-1}m_{au}(m^s_{uu})^{-1}m_{au}^\top m_{aa}^{-1} [\nabla(m_{au}\dot q_u)\dot q_u] \nonumber \\
&-&m_{aa}^{-1}m_{au}(m^s_{uu})^{-1}[C_{mu}(q_u,\dot q_u)\dot q_u+D_{mu}(q_u,\dot q)+\nabla V_u] \Big],
\label{bfS}
\end{eqnarray}
with the maps $C_{mu}:\rea^s \times \rea^s \mapsto \rea^{s \times s}$ and $D_{mu}:\rea^s \times \rea^n \mapsto \rea^s$ defined in \eqref{Cmuu} and \eqref{Dmu}, respectively. Notice that $S(q,\dot q)=0$ if there is no derivative action in the control.

To ensure that the control law \eqref{bark}---and consequently \eqref{pidcon1}---are well-defined we impose the full rank assumption. 

\begite
\item[{\bf A5.}] The controller tuning gains $k_e,k_a,k_u \in \rea$, $K_D \in \rea^{m \times m}$, $K_D \geq 0$ are such that 
$$
\det[  K(q_u)]\neq 0.
$$
\endite

It is clear that the analytic evaluation of the derivative term $\dot y_d$ considerably complicates the control expression. In some practical applications this term can be evaluated using an approximate differentiator of the form $\frac{bp}{p+a}$, with $p={d \over dt}$ and $a,b \in \rea_+$ designer chosen constants that regulate the bandwidth and gain of the filter. That is, a practical realisation of \eqref{pidcon1} is given by
\begequarr
\nonumber
 k_e {u} &  = & - K_P  y_d- {K_I }z_1  - K_D a(y_d-z_2) \\ 
\nonumber 
\dot z_1 & = & y_d,\;z_1(0)=z_1^0\\
\lab{appdif}
\dot z_2 & = & b(y_d-z_2).
\endequarr
It should be stressed that, as discussed in \cite{ASTHAG}, most practical PID controllers are implemented in this way. However, for applications that  require fast control actions---like the pendular system presented in Section \ref{sec6}---this approximation is not adequate. In this cases other, more advanced, techniques to compute time derivatives may be considered. 
%
\subsection{${\cal L}_2$-stability analysis}
\lab{subsec32}
As indicated in the introduction PIDs define input strictly passive maps therefore it is straightforward to prove ${\cal L}_2$-stability if it is wrapped around a passive output. Since $y_d$ is a linear combination of passive outputs \eqref{yd} and we have introduced the gain $k_e$ in \eqref{pidcon1}---with all these gains being {\em sign indefinite}---some care must be taken to ensure that we are dealing with the {\em negative} feedback interconnection of passive maps. This analysis is summarised in the two lemmata and the corollary below that establish the ${\cal L}_2$-stability of the closed-loop system represented in Fig. \ref{fig0}, where as it is customary we have added a external signal $d$ to define the closed-loop map.

\begin{figure}[htp]
 \centering
\includegraphics[width=.6\linewidth]{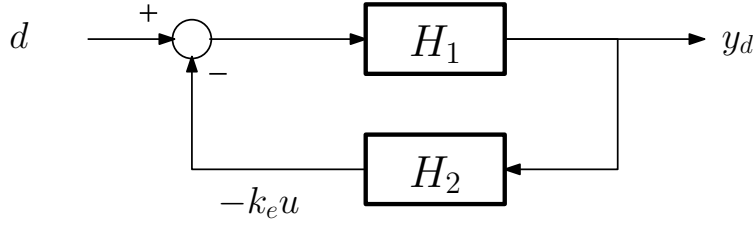}
 \caption{Block diagram representation of the system \eqref{lagr} in closed-loop with  \eqref{utau} and the PID  \eqref{outlag}, \eqref{pidcon1}, \eqref{yd} with an external signal $d$.}
 \label{fig0}
\end{figure}

\begin{lemma} \em
\lab{lem1}
Define an operator ${\bf H}_1: k_e u \mapsto y_d$ whose dynamics is given by the system \eqref{lagr} verifying Assumptions {\bf A1}-{\bf A4} in closed-loop with \eqref{utau} where $y_d$ is defined in \eqref{outlag}, \eqref{yd}.  Assume 
\begequ
\lab{sigcon}
\sign(k_e)=\sign(k_a)=\sign(k_u).
\endequ
The operator ${\bf H}_1$ is {\em passive}, that is, there exists $\beta_1 \in \rea$ such that
$$
\int_0^t k_e u^\top(s) y_d(s) ds \geq \beta_1,\;\forall t \geq 0.
$$
\end{lemma}
\begin{proof} 
The proof follows directly from Proposition \ref{pro1} noting that, because of \eqref{sigcon}, the function 
$$
H_1(q,\dot q):=k_e[k_a  H_a(q_u,\dot q)+k_u  H_u(q_u,\dot q_u)]
$$
is bounded from below and---due to \eqref{dotha}, \eqref{dothu}---it verifies $\dot H_1 = k_e u^\top y_d.$ 
\qed
\end{proof}

\begin{lemma} \em
\lab{lem2}
Define the linear time-invariant operator ${\bf H}_2:y_d \mapsto -k_eu$ defined by the PID controller \eqref{pidcon1}.  The operator ${\bf H}_2$ is {\em input strictly passive}. More precisely, there exists $\beta_2 \in \rea$ such that
$$
\int_0^t y_d^\top(s) [-k_e u(s)] ds \geq  {\lambda_{\min}( {K_P})}\int_0^t |y_d(s)|^2 ds + \beta_2,\;\forall t \geq 0,
$$
where $ \lambda_{\min}(\cdot)$ is the minimum eigenvalue.
\end{lemma}
\begin{proof}
Let us compute
\begequarrs
y_d^\top (-k_e u) & = &   y_d^\top  {K_P}  y_d+ y_d^\top {K_I}  z_1 +  y_d^\top{K_D } \dot y_d\\  
                     & \geq &    \lambda_{\min}( {K_P})  |y_d|^2+ \dot z_1^\top {K_I}  z_1 +  y_d^\top{K_D } \dot y_d.
\endequarrs
The proof is completed integrating the expression above and setting 
$$
\beta_2=-\|z_1(0)\|^2_{K_I} -  \|y_d(0)\|^2_{K_D}.
$$
\qed
\end{proof}

${\cal L}_2$-stability of the closed-loop system represented in Fig. \ref{fig0} is an immediate corollary of the two lemmata above, the Passivity Theorem \cite{DESVID} and the fact that \eqref{sigcon} ensures Assumption {\bf A5}---hence the feedback system is well-posed. 

\begin{corollary}\em
\lab{cor1}
Consider the system \eqref{lagr} verifying Assumptions {\bf A1}-{\bf A4} in closed-loop with  \eqref{utau} and  the PID  \eqref{outlag},  \eqref{pidcon1}, \eqref{yd} with an external signal $d$. Assume \eqref{sigcon} holds. The operator $d \mapsto y_d$ is ${\cal L}_2$-stable. More precisely, there exists $\beta_3 \in \rea$ such that
$$
\int_0^t |y_d(s)|^2 ds \leq  {1 \over \lambda_{\min}( {K_P})}\int_0^t |d(s)|^2 ds + \beta_3,\;\forall t \geq 0.
$$
\end{corollary}

The  ${\cal L}_2$-stability  analysis is of limited interest for the following two reasons. 
\begite
\item[{\bf (R1)}] ${\cal L}_2$-stability is a rather weak property. For instance, boundedness of trajectories is not guaranteed and the system can be destabilised by a constant external disturbance. Hence, we are interested in establishing a stronger property, {\em e.g.}, Lyapunov stability of a desired equilibrium.
\item[{\bf (R2)}] As explained in \cite{DONetal}, the gain $k_e$ is introduced in \eqref{pidcon1} to give more flexibility to the design. As shown in \cite{DONetal,ROMetal} and the example of Section \ref{sec6}, this feature is lost if we impose the condition \eqref{sigcon}, required by the analysis above. 
\endite

Before closing this subsection we make the following observation. It is easy to show that the approximated PID \eqref{appdif} defines an {\em output} strictly passive map, which is different from the   {\em input} strict passivity property of the original PID established in Lemma \ref{lem2}. Application of the Passivity Theorem proves now that the map $d \mapsto u$ is  ${\cal L}_2$-stable. Unfortunately, because of the presence of the integrator, nothing can be said about the map $d \mapsto y_d$. 
%
\section{Lyapunov Stabilisation via PID Control}
\lab{sec4}
%
In this section we prove that, under some additional {\em integrability} conditions on the inertia matrix, it is possible to ensure Lyapunov stability of a desired equilibrium position via PID-PBC.

In the sequel we will consider the system  \eqref{lagr} verifying Assumptions {\bf A1}-{\bf A4} in closed-loop with  \eqref{utau}. As shown in Appendix A it may be written as  \eqref{muuddqu},  \eqref{maaddqa} that we repeat here, in a slightly different form, for ease of reference
\begin{eqnarray}
m_{uu}\ddot q_u + m_{au}^\top \ddot q_a +C_{mu}(q_u,\dot q_u)\dot q_u +D_{mu}(q_u,\dot q) +\nabla V_u(q_u) & =& 0 
\label{muuddqu1} \\
m_{aa} \ddot q_a +m_{au} \ddot q_u+\nabla_{q_u} [m_{au}(q_u) \dot q_u]\dot q_u & = & u, 
 \label{maaddqa1}
\end{eqnarray}
with $C_{mu}(q_u,\dot q_u)$ and $D_{mu}(q_u,\dot q)$ given by  \eqref{Cmuu} and \eqref{Dmu}, respectively. We bring to the readers attention the important fact that 
\begequ
\lab{dmu1}
D_{mu}(q_u,0)=0.
\endequ
%
\subsection{An integrability assumption for equilibrium assignment}
\lab{subsec41}
A first step for Lyapunov stabilisation of a desired constant state is, obviously, to ensure that it is an {\em equilibrium} of the closed-loop. Since the system \eqref{lagr} is underactuated it is not possible to choose an arbitrary desired equilibrium, instead, it must be chosen as a member of the assignable equilibrium set. For the system  \eqref{muuddqu1},  \eqref{maaddqa1} this is given as
$$
{\cal E}:=\{(q,\dot q) \in \rea^n \times \rea^n\;|\; \dot q=0\;\mbox{and}\; \nabla V_u(q_u)=0\},
$$ 
where we have used the property \eqref{dmu1}. An additional difficulty stems from the fact that the signal $y_d$ is equal to zero for {\em all} constant values of $q$. Hence, the PID control \eqref{pidcon1}, \eqref{yd1} does not allows us to impose an assignable equilibrium to the closed-loop. To overcome this obstacle  we add a condition to the system inertia matrix to be able to express the integral term of the PID, {\em i.e.}, the signal $z_1$, as a function of $q$. For, we assume the following integrability condition.

\begite
\item [{\bf A6.}]  The rows of $m_{au}(q_u)$, denoted $(m_{au}(q_u))^i$, are gradient vector fields, that is,
$$
\nabla(m_{au})^i=[\nabla (m_{au})^i]^\top, \;\forall i \in \bar{m}.
$$
\endite

The latter assumption is equivalent to the existence of a mapping $ V_N:\RE^{s} \to \RE^m$ such that
\begequ
\lab{dotvn}
\dot{ V}_N=m_{aa}^{-1} m_{au}(q_u)\dot{q}_u.
\endequ
Replacing the latter in \eqref{yd1} and this, in its turn, in \eqref{pidcon1} yields
\begequ
\lab{dotz1}
\dot z_1 = y_d=k_a \dot q_a {+ (k_a - k_u)} \dot V_N(q_u).
\endequ
Integrating \eqref{dotz1} with $z_1(0)=z_1^0$ we finally get
\begequ
\lab{z11}
z_1(t)= k_a q_a(t) {+ (k_a - k_u)} V_N(q_u(t))  + \kappa,
\endequ
where
$$
\kappa:= z_1^0-k_a q_a(0)-(k_a-k_u)V_N(q_u(0)).
$$
In this way we have achieved the desired objective of adding a term dependent on $q$ in the control signal.

We have the following simple propostion.

\begin{proposition} \em
\label{pro2}
Consider the underactuated mechanical system \eqref{lagr} satisfying Assumptions {\bf A1}--{\bf A4} and {\bf A6}, together with \eqref{utau} and the PID controller \eqref{pidcon1} and \eqref{yd1} verifying the well-posedness Assumption {\bf A5}. Fix $q_u^\star \in \rea^s$ such that
\begequ
\lab{nabvu1}
\nabla V_u(q_u^\star)=0
\endequ
and
\begequ
\lab{z10}
z_1^0= k_a [q_a(0) - q_a^*]+(k_a-k_u)[V_N(q_u(0))- V_N(q_u^\star)],
\endequ
where $q_a^\star \in \rea^m$ is {\em arbitrary}. Then, $(q,\dot q)=(q^\star,0)$ is an equilibrium point of the closed-loop system.  
\end{proposition}

\begin{proof}
First, notice that \eqref{nabvu1} ensures $(q^\star,0) \in {\cal E}$ for any  $q_a^\star \in \rea^m$. Evaluating the control signal \eqref{pidcon1}, \eqref{yd1}  at  $\dot q=0$ yields
\begequarr
\nonumber
 u|_{\dot q=0} & = & -{K_I \over k_e} z_1 \\
                       & = & {K_I \over k_e}\{k_a (q_a-q_a^\star) {+ (k_a - k_u)} [V_N(q_u)-V_N(q_u^\star)]\},
\lab{uatequ}
\endequarr
where we have used \eqref{z11} and \eqref{z10} to get the second identity. The proof is completed replacing the expression above, which is zero at $q=q^\star$, in \eqref{maaddqa1} and setting $(q,\dot q)=(q^\star,0)$. 
\end{proof}

%
\subsection{Construction of the Lyapunov function}
\lab{subsec42}
Define the function $U:\rea^n \times \rea^n \times \rea^m \to \rea$
\begin{equation}
U(q,\dot q,z_1) :=   k_e[k_a  H_a(q_u,\dot q)+k_u  H_u(q_u,\dot q_u)]+ {1 \over 2} \| y_d\|^2_{K_D}+ {1 \over 2}\|z_1\|_{K_I}^2,
\lab{u}
\end{equation}
with $y_d$ given in \eqref{yd1}. From \eqref{dotha}-\eqref{yd} it is straightforward to show that 
\begequ
\lab{dotu}
\dot U = -\| y_d\|^2_{K_P} \leq 0.
\endequ
A LaSalle-based analysis \cite{KHA} allows us to establish from \eqref{dotu} some properties of the system trajectories, for instance to conclude that that $y_d(t) \to 0$---see Subsection \ref{subsec52}. However, as indicated in the introduction our objective in the paper is to prove Lyapunov stability. Towards this end, it is necessary to construct a Lyapunov function for the closed-loop system, which is done finding a function $H_d:\rea^n \times \rea^n \to \rea$ such that
\begequ
\lab{uequhd}
U(q,\dot q,z_1) \equiv H_d(q,\dot q).
\endequ
In view of \eqref{dotu} and \eqref{uequhd} we have that $H_d(q(t),\dot q(t))$ is a non-decreasing function therefore it will be a {\em bona fide} Lyapunov function if we can ensure it is {\em positive definite}. 

To establish the latter we observe from \eqref{yd1} that  the first three terms of \eqref{u} can be written as
\begequ
\lab{rhs}
k_e[k_a  H_a(q_u,\dot q)+k_u  H_u(q_u,\dot q_u)]+ {1 \over 2} \| y_d\|^2_{K_D}=\hal \dot q^\top M_d(q_u) \dot q + k_e k_u V_u(q_u),
\endequ
with
\begequ
\label{mdph}
M_d(q_u)\hspace{-1mm}:= \hspace{-1mm}\left[ \begarr{cc}   A(q_u)   &  k_e k_a m_{au}^\top(q_u)+k_a(k_a-k_u) m_{au}^\top(q_u) m_{aa}^{-1}K_D   \\  k_e k_a {m_{au}} + {k_a} (k_a-k_u)K_D m_{aa}^{-1}m_{au}(q_u)   &  k_e k_a m_{aa} +k_a^2 {K_D} \endarr \right]
\endequ
and
$$
A(q_u):=k_e k_u m^s_{uu}(q_u) + k_ek_a m_{au}^\top(q_u)m_{aa}^{-1}m_{au}(q_u)  + (k_a-k_u)^2m_{au}^\top (q_u) m_{aa}^{-1}K_D m_{aa}^{-1}m_{au}(q_u).
 $$
Unfortunately, the right hand side of \eqref{rhs} {\em does not} depend on $q_a$ and, consequently, cannot be a positive definite function of the full state. At this point we invoke Assumption {\bf A6} and replace \eqref{z11} in \eqref{u} to get
 \begequarr
\lab{tilhd}
 H_d(q, \dot q)= \frac{1}{2}   \dot q^\top  M_d(q_u) \dot q+  V_d(q).
\endequarr
with 
$$
V_d(q):=k_ek_uV_u(q_u) + \frac 12 || k_aq_a {+ (k_a-k_u)}   V_N(q_u)+\kappa||^2_{K_I},
$$
where we notice the presence of the term $\kappa$, which contains the initial condition of the integrator $z_1(0)$.\\

Before closing this subsection we elaborate on remark {\bf (R2)} of Subsection \ref{subsec32} and attract the readers attention to the presence of the term $k_ek_u$ in the first right hand term of \eqref{vdph}. In most pendular systems $q_u$ represents the pendulum angle and the potential energy function $V_u(q_u)$ has a {\em maximum} at its upward position, which is an unstable equilibrium. If the control objective is to  swing up the pendulum and stabilize this equilibrium this maximum can be transformed into a minimum choosing the controller gains such that $k_ek_u < 0$. See the example in Subsection \ref{subsec62}.
%
\subsection{Lyapunov stability analysis}
\lab{subsec43}
The final step in our Lyapunov analysis is to ensure that $H_d(q,\dot q)$ is {\em positive definite}. For, we first select the integrators initial conditions as \eqref{z10} that yields
\begin{equation}\label{vdph}
V_d(q):=k_ek_uV_u(q_u) + \frac 12 \| k_a(q_a-q_a^\star) {+ (k_a-k_u)} [V_N(q_u)-V_N(q_u^\star)]\|^2_{K_I},
\end{equation}
this, together with \eqref{nabvu1}, ensures $V_d(q)$ has a critical point at the desired position $q_\star \in \rea^n$. Second, we make the following final assumption.
\begite
\item[{\bf A7.}]  The controller tuning gains $k_e,k_a,k_u \in \rea$, $K_D, K_I \in \rea^{m \times m}$, $K_I>0,K_D \geq 0$ are such that the matrix $M_d(q_u)$ defined in  \eqref{mdph} is positive definite and the function $V_d(q)$ defined in \eqref{vdph} has an isolated minimum at $q_\star \in \rea^n$.
\endite

We are in position to present the first main result of the note, whose proof follows from \eqref{dotu}, \eqref{uequhd} and standard Lyapunov stability theory. 

\begin{proposition} \em
\label{pro3}
Consider the underactuated mechanical system \eqref{lagr} satisfying Assumptions {\bf A1}--{\bf A4} and {\bf A6}, together with \eqref{utau} and the PID controller \eqref{pidcon1} and \eqref{yd1}, verifying Assumptions {\bf A5} and {\bf A7}, with $z_1(0)$ given in  \eqref{z10}. 
\begite
\item[(i)] The closed--loop system has a {\em stable} equilibrium at the desired point $(q,\dot q)=(q_\star,0)$ with Lyapunov function \eqref{tilhd} with $ M_d(q_u)$ and $ V_d(q_u)$ defined in \eqref{mdph} and   \eqref{vdph}, respectively.  
\item[(ii)] The equilibrium is  {\em asymptotically} stable if the signal $ y_d$ is a {\em detectable} output for the closed--loop system. 
\item[(iii)] The stability properties are {\em global} if $ V_d(q)$ is radially unbounded.
\endite
\end{proposition}

It may be argued that Proposition \ref{pro3} imposes a particular initial condition to the controller state $z_1$ making the result ``trajectory dependent" and somehow fragile. In this respect notice that fixing the initial condition of the integrator state is equivalent to  fixing a constant additive term in a {\em static} state-feedback implementation of the control---a common practice in nonlinear control designs. Moreover, no claim concerning the transient stability is made here, in which case the ``trajectory dependence" of the controller renders the claim specious. See Remark 10 and the corresponding sidebar of \cite{ORTPAN}.   

\section{Extensions}
\lab{sec5}
In this section we present the following two extensions to Proposition \ref{pro3}.
\begite
\item The proof that the cancellation of the potential energy term $V_a(q_a)$ of the preliminary feedback \eqref{utau} can be eliminated for a certain class of potential energy functions.
\item The relaxation  of the conditions imposed by Assumption {\bf A7} on the PID tuning gains.
\endite
%
\subsection{Removing the cancellation of $V_a(q_a)$}
\lab{subsec51}

As clear from the derivations above the key step for the design of the PID-PBC is to prove that, even without the cancellation of the term $V_a(q_a)$, the mappings  $\tau \mapsto   y_u$ and  $\tau \mapsto   y_a$ are {\em passive} with suitable storage functions. This fact is stated in the proposition below whose proof is given in Appendix B and requires the following assumption.

\begite
\item[{\bf A8.}]  The function $V_a(q_a)$ is of the form
\begin{equation}
\label{Vaqa}
V_a(q_a)=s_a^\top q_a + c_0,
\end{equation}
with $s_a \in \rea^m$ and $c_0 \in \rea$.
\endite
\begin{proposition} \em
\label{pro4}
Consider the underactuated mechanical system  \eqref{lagr} satisfying Assumptions {\bf A1}--{\bf A4}, {\bf A6} and {\bf A8}. The operators $\tau \mapsto   y_u$ and  $\tau \mapsto   y_a$ are {\em passive} with storage functions  
\begequarr
\bar H_u(q_u,\dot {q}_u) & := &  \hal \dot q_u^\top m_{uu}^s(q_u)\dot q_u + V_u(q_u) - V_0(q_u)
 \lab{newstofun1a} \\
\bar H_a(q,\dot{q}) & := & \hal  \dot q^\top  M_a(q_u) \dot q + V_a(q_a)+V_0(q_u),
\lab{newstofun2a}
\endequarr
where
\begequ
\label{V0}
V_0(q_u(t)) = s_a^\top  \int^t_0  y_u(s) ds + c_0,
\endequ
where $c_0 \in \rea$.
\end{proposition}

In contrast to Proposition \ref{pro3}, the proposition above does not include the nonlinearity cancellation due to the feedback \eqref{utau}. On the other hand, we impose Assumption {\bf A8} and the integrability Assumption {\bf A6}---the latter is, in any case, required for the Lyapunov stability analysis of the closed-loop. Notice also that, under Assumption {\bf A6}, $V_0(q_u)$ is well-defined.

\subsection{Convergence analysis via LaSalle's invariance principle}
\lab{subsec52}
%
In this subsection we remove Assumption {\bf A7} and we perform a convergence analysis using the following assumptions.

\begite
\item[{\bf A9.}] The system is {\em strongly inertially coupled} \cite{SPO}, that is,
$$
\rank[m_{au}(q_u)]=s,
$$ 
and the function $\nabla V_u(q_u)$ is {\em injective}.
\endite 

As shown in the proof of the proposition below, which is given in Appendix C, Assumption {\bf A9} is required to complete the convergence analysis.

\begin{proposition} \em
\label{pro5}
Consider the underactuated mechanical system \eqref{lagr} verifying Assumptions {\bf A1}-{\bf A4}, {\bf A6} and {\bf A9} in closed-loop with \eqref{utau} and the PID controller \eqref{pidcon1} and \eqref{yd1},  verifying the well-posedness Assumption {\bf A5}, with $z_1(0)$ given in  \eqref{z10}. All {\em bounded} trajectories of the closed loop system verify 
$$
\lim_{t \to \infty}\lef[{c}q(t) \\ \dot q(t) \rig]=\lef[{c} q_\star \\ 0 \rig].
$$
\end{proposition}

It should be underscored that Assumption {\bf A7}, which imposes constraints on the PID-PBC gains to ensure positive definiteness of the function $H_d(q,\dot q)$---and, consequently, Lyapunov stability of the desired equilibrium---is conspicuous by its absence. The prize that is paid for this relaxation  is that there is no guarantee that trajectories remain bounded. However, there are cases where boundedness of trajectories  can be established invoking other (non Lyapunov-based) considerations---see, for instance, the proof of Proposition 9 in \cite{VENetal}.  

On the other hand, we impose stronger conditions on the coupling between the actuated and unactuated dynamics captured in Assumption {\bf A9}. As discussed in \cite{SPO} the condition strong inertial coupling is, essentially, a controllability condition and it requires that $s\le m$, {\em i.e.}, the number of actuated coordinates is larger or equal than the unactuated ones.  In Exercise {E10.17} of \cite{BULLEW} it is shown that this assumption is not coordinate-invariant nor is it related to stabilisability of the system. In the sense that there are strongly inertially coupled systems that cannot be stabilised using any kind of feedback. The assumption of injectivity of $\nabla V_u(q_u)$ seems, unfortunately, unavoidable without imposing conditions on $V_d(q)$.
  
Propositions \ref{pro4} and \ref{pro5} can be combined yielding a robust controller---that does not cancel the term $V_a(q_a)$---and does not require Assumption {\bf A7}. This case is omitted for brevity.
%
\section{EXAMPLES}
\lab{sec6}
%
In this section we apply the proposed PID-PBC  to linear mechanical systems and the well-known cart-pendulum on an inclined plane system. 
\subsection{Linear mechanical systems}
\lab{subsec61}
%
For linear mechanical systems verifying Assumption {\bf A1} the dynamical model  \eqref{lagr} reduces to
\begequ 
\lab{lti}
M \ddot{q} + S q =\lef[{c} 0_{{s}\times m} \\ I_m  \rig] \tau,
 \endequ
 where $M>0$ is constant and $S=S^\top$ defines the quadratic potential energy $V(q)=\hal q^\top S q$. Assumptions {\bf A2} and {\bf A3} are, clearly, satisfied. To comply with Assumption {\bf A4} the matrix $S$ is of the form
 $$
 S=\lef[{cc} S_u & 0_{s \times m} \\ 0_{m \times s} & S_a \rig]
 $$
 where  $S_u \in \rea^{s \times s}, S_a \in  \rea^{{m}\times m}$.    The inner-loop control \eqref{utau} is given as $u=\tau-S_a q_a$.  Replacing this signal into \eqref{lti} and using the fact that $S_u q_u^\star=0$ we get
 $$
 M \ddot{\tilde q} + \lef[{cc} S_u & 0_{s \times m} \\ 0_{m \times s} & 0_{m \times m} \rig] \tilde q =\lef[{c} 0_{{s}\times m} \\ I_m  \rig] u,
 $$
 where  $\tilde q:=q - q^\star$ are the position errors.
 
 Some simple calculations show that the PID-PBC \eqref{pidcon1}, \eqref{yd1} may be written as
 $$
k_e  u = - (K_D s^2 + K_P s + K_I)  \left[m_0\tilde q_u + {k_a} \tilde q_a \right],
 $$
where we defined the matrix 
$$
m_0:={(k_a - k_u)}m_{aa}^{-1}m_{au} \in \rea^{m \times s},
$$
and---abusing notation---we mix the Laplace transform and time domain representations. Notice that the constant term $\kappa$ of   \eqref{pidcon1}, which is retained only in the integral term of the control, is incorporated into the definition of the error signals. 
 
The final closed-loop system may be written as
 $$
 \left\{ \lef[{cc} m_{uu} &  m^\top_{au} \\  m_{au}+ {1 \over k_e} K_D m_{0} &{k_a \over k_e}  K_D  \rig]s^2
+  \lef[{cc}  0_{s \times s} &  0_{s \times m}\\ {1 \over k_e} K_P m_{0} &{k_a \over k_e}  K_P  \rig]s
+  \lef[{cc} S_u &  0_{s \times m}\\ {1 \over k_e} K_I m_{0} &{k_a \over k_e}  K_I  \rig] 
 \right\} \tilde q=0.
 $$ 
The closed-loop system will be asymptotically stable {\em if and only} if the determinant of the polynomial matrix in brackets is a Hurwitz polynomial. 

The Lyapunov function \eqref{tilhd} used in Proposition \ref{pro3} is of the form
$$
H_d(\tilde q, \dot {\tilde q})= \frac{1}{2}   \dot {\tilde q}^\top  M_d \dot {\tilde q}+ \hal \tilde q^\top  \lef[{cc} k_e k_u S_u + m_0^\top K_I m_0 &  k_a  m_0^\top K_I\\ k_a  K_I m_{0} & k_a^2 K_I  \rig] \tilde q, 
$$
with the constant matrix $M_d$ given by  \eqref{mdph}. Positivity of this function is, clearly, only {\em sufficient} for asymptotic stability of the closed-loop system.

\subsection{Cart-pendulum on an inclined plane}
\lab{subsec62}
%
In this subsection we consider the cart-pendulum on an inclined plane system depicted in Fig. \ref{fig1}. The objective is to stabilize a desired position of the cart as well as the pendulum at the upright position applying the PID-PBC of Proposition \ref{pro4}.
 
\begin{figure}[htp]
 \centering
\includegraphics[width=0.69\linewidth]{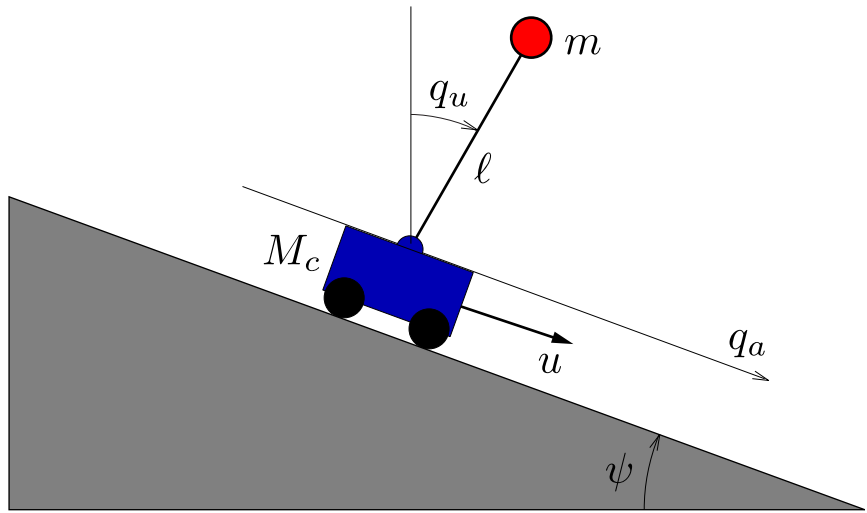}
 \caption{The cart-pendulum on an inclined plane.}
 \lab{fig1}
\end{figure}

The dynamics of the system has the form \eqref{lagr} with $n=2$, $q_a$ the position of the car, $q_u$ the angle of the pendulum with respect to the up-right vertical position and $u$ a force on the cart.  The inertia matrix is
\begin{equation*}
 M (q_u)= \left[ \begarr{cc} m \ell^2 & m \ell \cos(q_u-\psi) \\ m \ell \cos(q_u-\psi) & M_c+m \endarr \right],
 \end{equation*}
with  $M_c,m$ the masses of the cart and pendulum, respectively, $\ell$ the pendulum length and $\psi$ the angle of inclination of the plane. The potential energy function is
$$
V (q)= mg\ell \cos(q_u) - (M_c+m)g\sin(\psi)q_a,
$$
and the input matrix is $G=\col(0,1)$. The desired equilibrium is $(0, q_a^\star)$ with $q_a^\star \in \rea$, which is the only constant assignable equilibrium point.

This system clearly satisfies Assumptions {\bf A1-A4} and {\bf A8} with $s_a= -(M_c+m)g\sin(\psi)$ and $c_0=0$.  
Applying Proposition \ref{pro1} we identify the cyclo--passive outputs as
\begin{align*}
 y_a &= \dot q_a + \frac{m\ell}{M_c+m} \cos(q_u-\psi) \dot q_u \\
 y_u &= - \frac{m\ell}{M_c+m} \cos(q_u-\psi) \dot q_u. 
\end{align*}
The signal $y_d$ defined in \eqref{yd} takes the form
$$
 y_d=k_a \dot q_a + (k_a-k_u) \frac{m\ell}{M_c+m} \cos(q_u-\psi) \dot q_u.
$$
Assumption {\bf A6} is also satisfied with 
$$
 V_N(q_u)=  \frac{m\ell}{M_c+m} \sin(q_u-\psi).
$$
Finally, from Proposition \ref{pro2} the PID controller is given by \eqref{pidcon1} where the integral term \eqref{z11} takes the form 
$$
z_1 =  k_aq_a+(k_a-k_u) \frac{m\ell}{M_c+m} \sin(q_u-\psi)+\kappa,
$$
and
\begin{align*}
S(q,\dot q)& =  -k_u K_D \Bigg\{ -\frac{m\ell}{M_c+m} \sin(q_u-\psi) \dot q_u^2  +\mathcal{N}(q_u) \left[ -m\ell \sin(q_u-\psi) \dot q_u^2 - (M_c+m) g \sin(\psi) \right]  \nonumber \\
   +& \frac{m^2\ell^2 g}{M_c+m} (m_{uu}^s)^{-1} \cos(q_u-\psi)  \sin(q_u) \Bigg\} +  k_aK_D g \sin(\psi), \\
\mathcal{N}(q_u) &= \frac{m \cos^2 (q_u-\psi)}{(M_c+m) \left[M_c+m-m\cos^2(q_u-\psi) \right]},\\
K(q_u) &= - \frac{k_u K_D m \left[\cos(q_u-\psi)\right]^2}{(M_c+m)\left(M_c+m \left[\sin(q_u-\psi)\right]^2\right)} +k_e+\frac{k_aK_D}{(M_c+m)},  \\
\kappa &= -k_aq_a^\star +(k_a-k_u) \frac{m\ell \sin(\psi)}{M_c+m}. 
\end{align*}

The  parameters and initial conditions used in the simulations  have been chosen according to \cite{BLOetal} and they are given as follows:  $m= 0.14$ $kg$, $M_c=0.44$ $kg$, $l=0.215$ $m$,  $\psi=20$ degrees,  $q(0)=(20$deg$,-0.6$m) and $\dot q(0)=0$. The desired equilibrium is set to $q_a^\star=0$ $m$ for $t \in (0,5)$ and to $q_a^\star=-0.3$ $m$ for $t \in (5,10)$, with $q_u^\star=0$ always.

The gains of the PID-PBC \eqref{pidcon1} are chosen  as  $K_D=0.1$, $K_P=1$ and $K_I=2$. Using these values we present three set of simulations where we change one-by-one the gains $k_a$, $k_u$ and $k_e$ while keeping the PID-PBC gains unaltered---always satisfying  Assumptions {\bf A5} and {\bf A7}. Variations of the PID-PBC gains were also considered but their effect was less informative than changing the gains $k_a,k_u$ and $k_e$ In all cases we present the transient behavior of $q$, $\dot q$, $u$ and the factor $K(q_u)$.  Figs. \ref{q_ka}-\ref{uK_ka} correspond to variations of $k_a$ with $k_u=-500$ and $k_e=5$.  In Figs. \ref{q_ku}-\ref{uK_ku} we change now $k_u$ with  $k_a= 50$ and $k_e=5$. Finally, Figs. \ref{q_ke}-\ref{uK_ke} correspond to variations of $k_e$ with $k_a= 50$ and $k_u=-500$. In all cases, the desired regulation objective is achieved very fast with a reasonable control effort. 

These plots should be compared with Fig. 5 in  \cite{BLOetal} where the transient peaks are much larger and they take over $100$ $s$ to die out. Unfortunately, the plots of the control signal are not shown in  \cite{BLOetal}, but given the magnitudes selected in the controller it is expected to be much larger than the ones resulting from the PID-PBC.

\begin{figure}[htp]
 \centering \hspace{-1.1cm}
\includegraphics[width=1.06\linewidth]{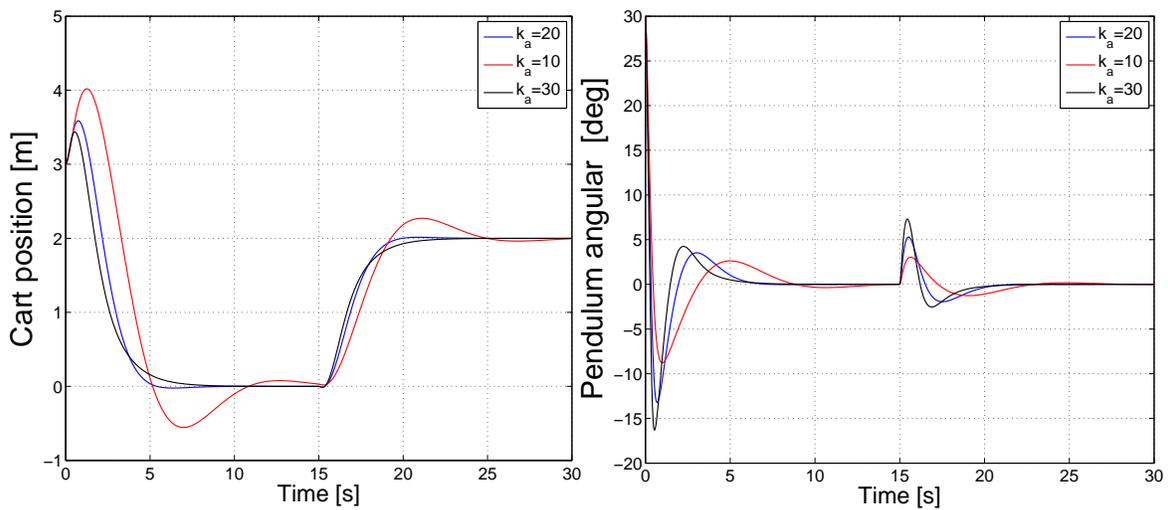}
 \caption{Time histories of the position of the cart $q_a(t)$ and angle of the pendulum $q_u(t)$.}
 \label{q_ka}
\end{figure}

\vspace{-.7cm}
\begin{figure}[htp]
 \centering \hspace{-1.1cm}
\includegraphics[width=1.06\linewidth]{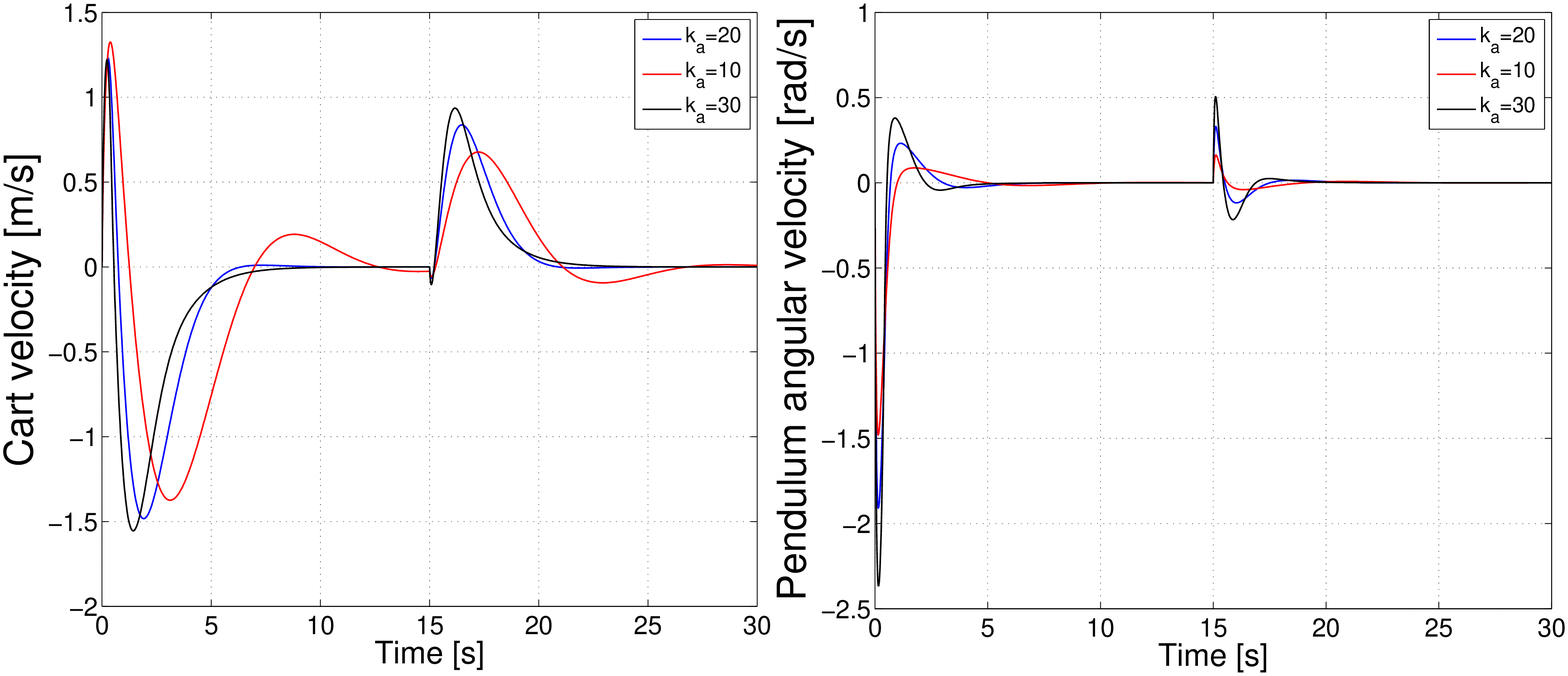}
 \caption{Time histories of the velocity of the cart $\dot q_a(t)$ and angular velocity of the pendulum $\dot q_u(t)$.}
 \label{dq_ka}
\end{figure}

\vspace{-.9cm}
\begin{figure}[htpi]
 \centering \hspace{-1.1cm}
\includegraphics[width=1.06\linewidth]{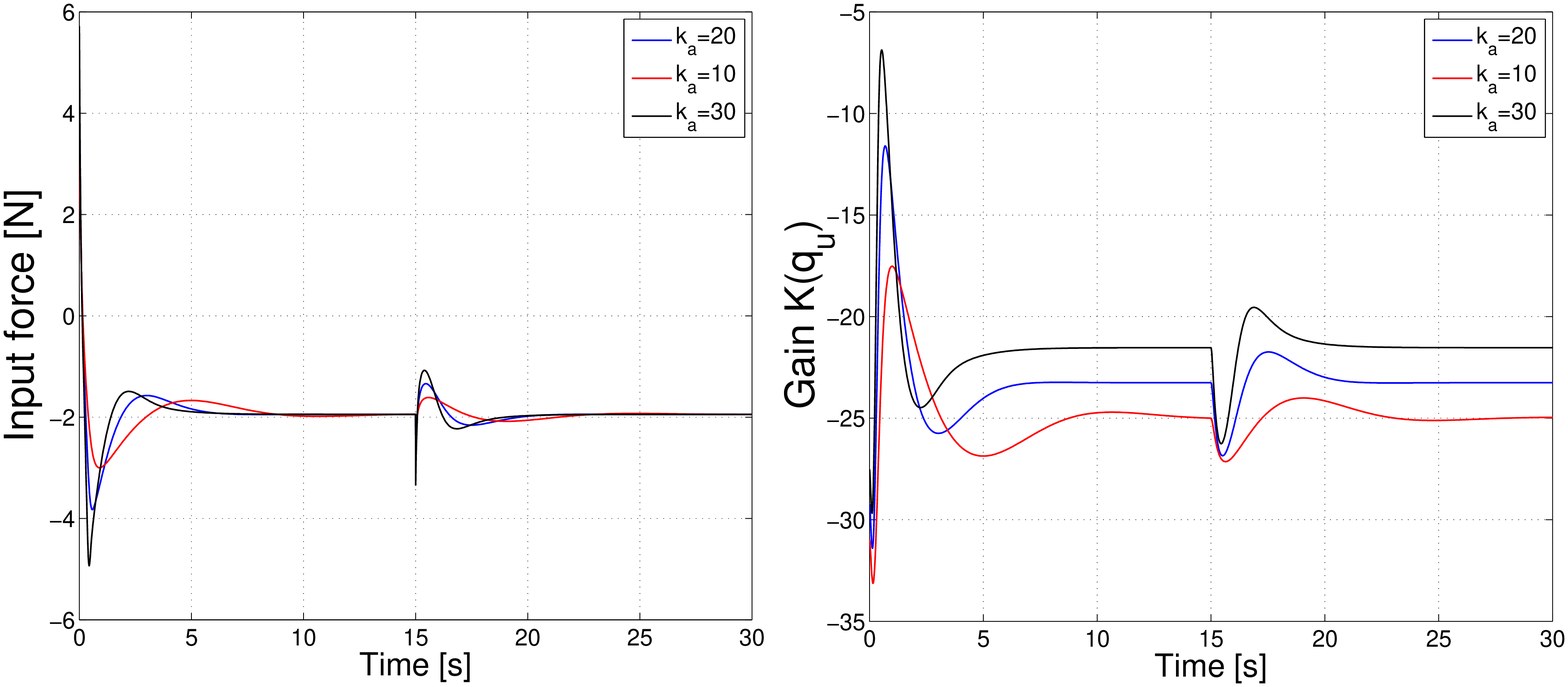}
 \caption{Time histories of the input force $u(t)$ and the nonlinear gain $K(q_u)$.}
 \label{uK_ka}
\end{figure}

\vspace{-1.2cm}

\begin{figure}[htpi]
 \centering \hspace{-1.1cm}
\includegraphics[width=1.06\linewidth]{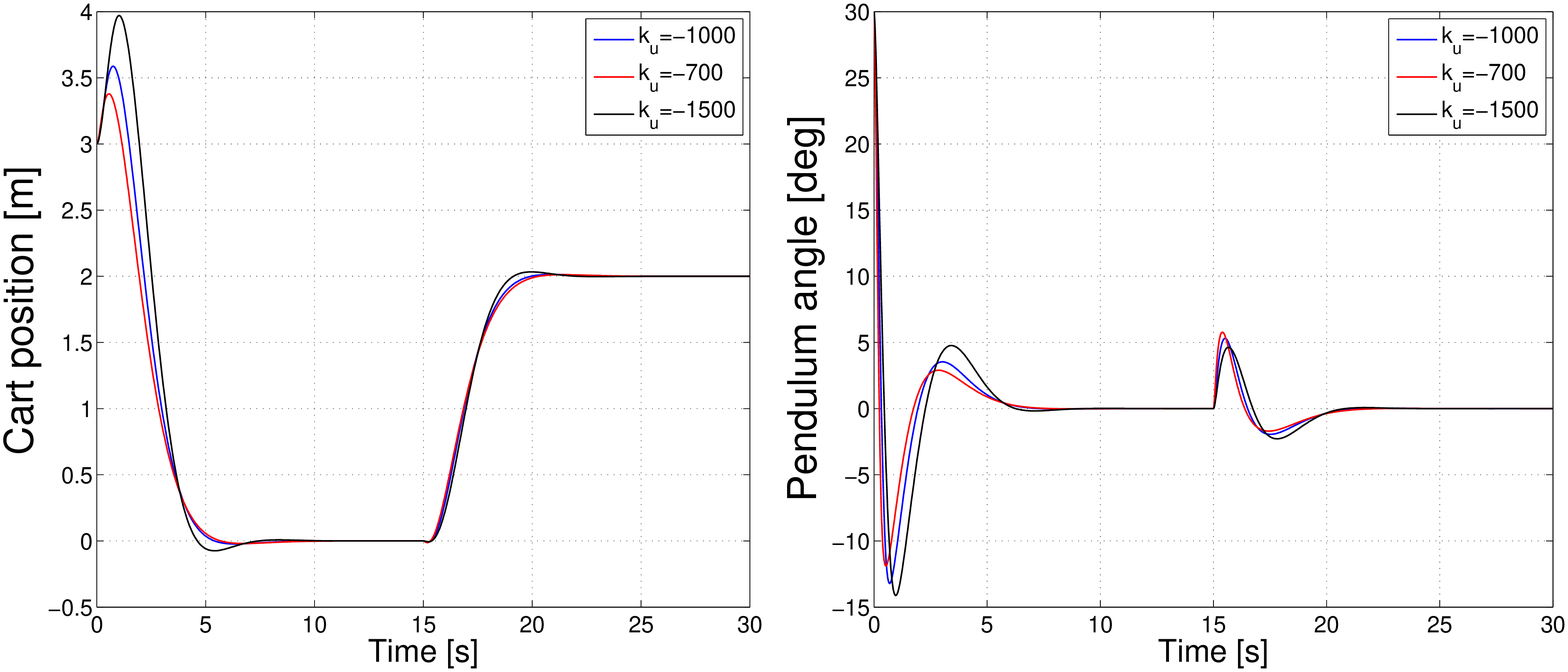}
 \caption{Time histories of the position of the cart $q_a(t)$ and angle of the pendulum $q_u(t)$.}
 \label{q_ku}
\end{figure}

\begin{figure}[htpi]
 \centering \hspace{-1.1cm}
\includegraphics[width=1.04\linewidth]{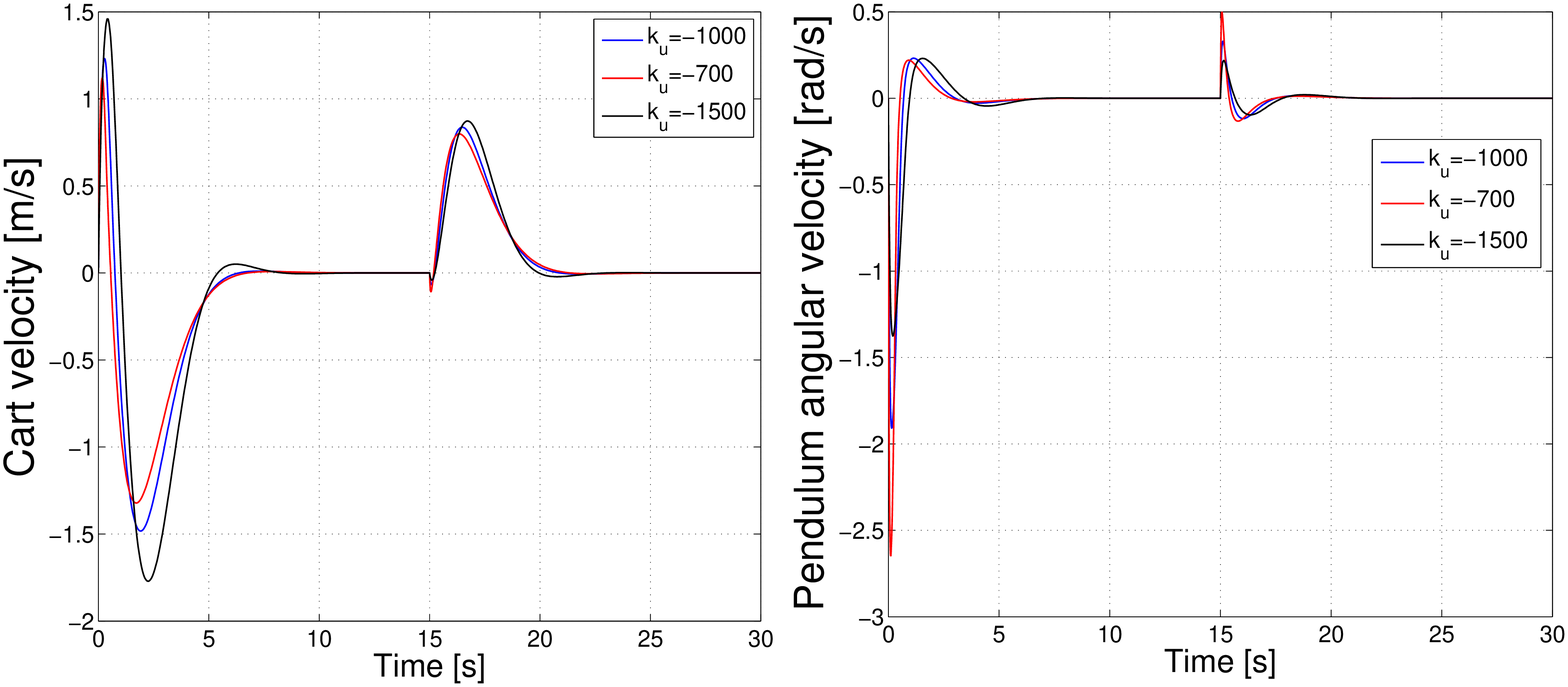}
 \caption{Time histories of the velocity of the cart $\dot q_a(t)$ and angular velocity of the pendulum $\dot q_u(t)$.}
 \label{dq_ku}
\end{figure}

\begin{figure}[htpi]
 \centering \hspace{-1.1cm}
\includegraphics[width=1.06\linewidth]{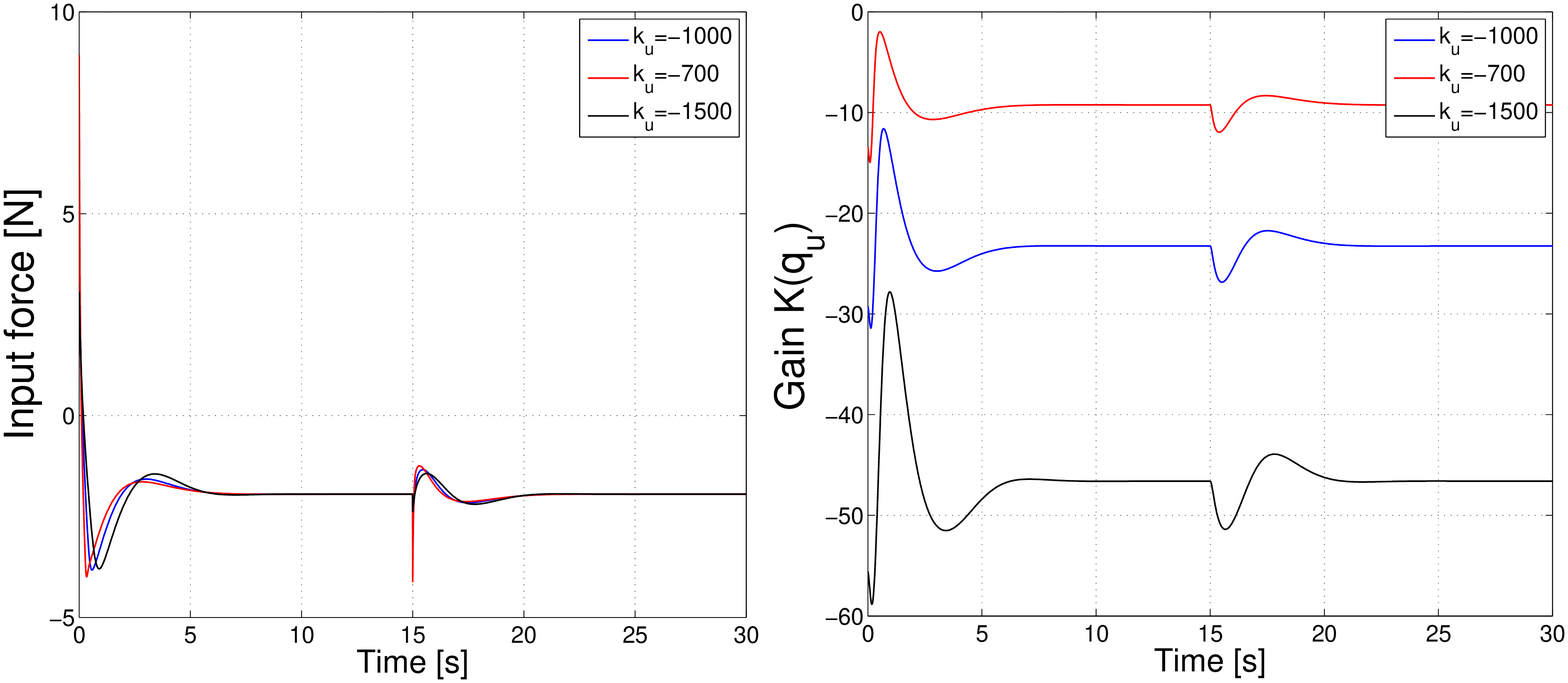}
 \caption{Time histories of the input force $u(t)$ and the nonlinear gain $K(q_u)$.}
 \label{uK_ku}
\end{figure}


\begin{figure}[htpi]
 \centering  \hspace{-1.1cm}
\includegraphics[width=1.04\linewidth]{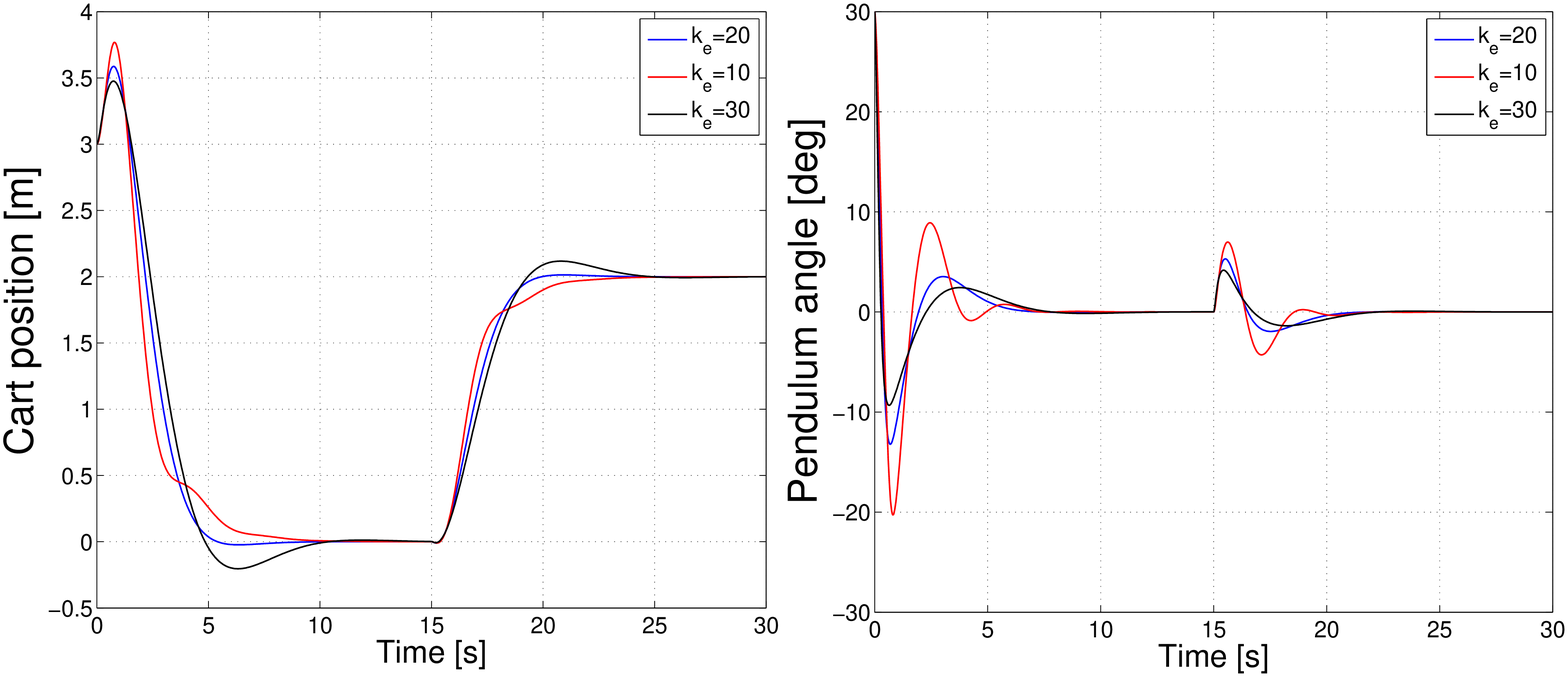}
 \caption{Time histories of the position of the cart $q_a(t)$ and angle of the pendulum $q_u(t)$.}
 \label{q_ke}
\end{figure}

\begin{figure}[htpi]
 \centering  \hspace{-1.1cm}
\includegraphics[width=1.06\linewidth]{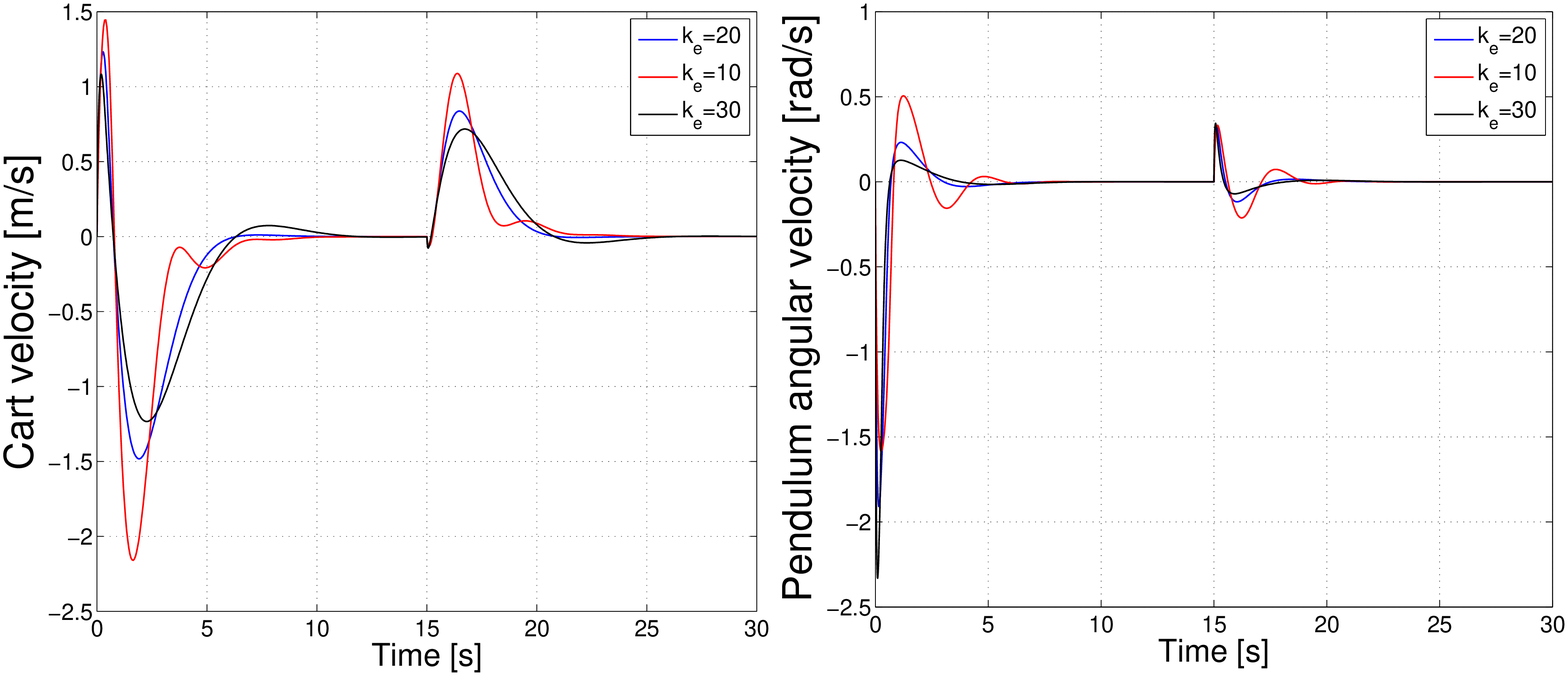}
 \caption{Time histories of the velocity of the cart $\dot q_a(t)$ and angular velocity of the pendulum $\dot q_u(t)$.}
 \label{dq_ke}
\end{figure}

\begin{figure}[htpi]
 \centering \hspace{-1.1cm}
\includegraphics[width=1.06\linewidth]{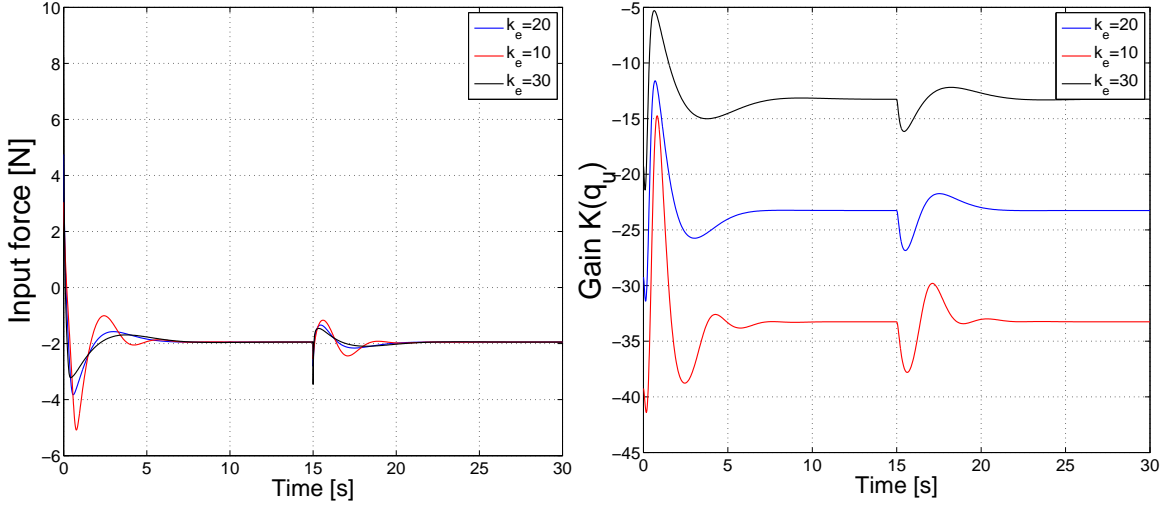}
 \caption{Time histories of the input force $u(t)$ and the nonlinear gain $K(q_u)$.}
 \label{uK_ke}
\end{figure}

\vspace{1.8cm}
Finally, we present in Fig. \ref{iAnimation} a serie of captures of a video animation of the cart-pendulum with initial conditions $(q(0),\dot q(0))= (20$deg, -$0.6$m), $\dot q (0)=0$ and desired equilbirum at the origin. The controller gains were selected as follows: $k_a= 50$, $k_u=-450$, $k_e=5$, $K_D=0.1$, $K_P=1$ and $K_I=2$. As it can be seen in the animation the PID-PBC ensures very good performance while satisfying Assumptions {\bf A5} and {\bf A7}.


\begin{figure}[ht]
 \centering \hspace{-5mm}
\includegraphics[width=.91\linewidth]{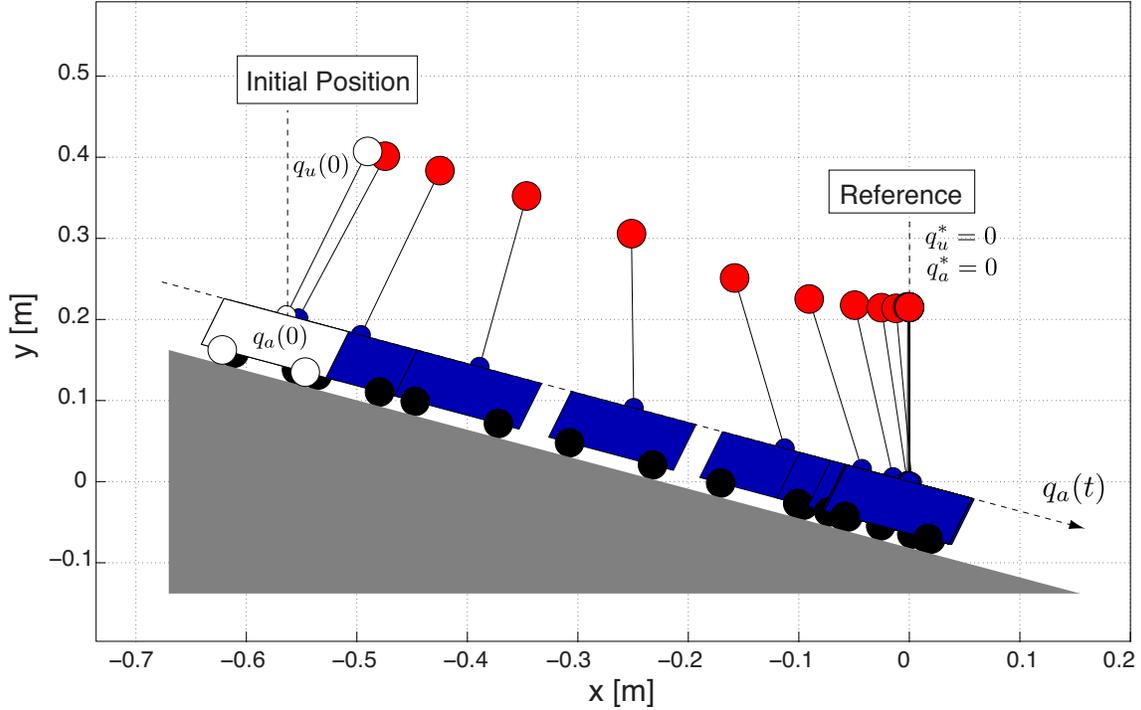}
\caption{Captures of a video animation of the cart-pendulum on an inclined plane. The full-video animation can be watched in {\tt https://youtu.be/CGInoXkR0FA}.}
\label{iAnimation}
\end{figure}

%
\section{Conclusions and Future Research}
\lab{sec7}
%
We have identified in this paper a class of underactuated mechanical systems whose constant position can be stabilized with a linear PID controller. It should be underscored that in view of the freedom in the choice of the {\em signs} of the constants entering into the control design, {\em i.e.}, $k_e,k_a$ and $k_u$, the proposed PID is far from being standard. Given the popularity and simplicity of this controller, and the fact that the class contains some common benchmark examples, the result is of practical interest. Moreover, from the theoretical view point, the {\em performance} of the PID controllers sometimes parallels the one of total energy shaping controllers like IDA PBC. For instance, it has been proved that in some benchmark examples (like the cart-pendulum system and the inertia wheel),  the desired equilibrium has the same estimated domain of attraction for both designs \cite{DONetal,ROMetal}.  See also the example of Subsection \ref{subsec62} where it is shown, via simulations, that the transient performance of the PID-PBC is far superior to the one of the total energy shaping controller reported in \cite{BLOetal}. 

Besides the usual Lyapunov stability analysis, that imposes some constraints on the PID tuning gains to shape the energy function, a LaSalle-based study of attractivity has been carried out under strictly weaker conditions on these gains, but imposing the stronger Assumption {\bf A9} on the system. An additional contribution of the paper is the proof that it is possible to obviate the cancellation of the actuated part of the potential energy, provided this is described by an affine function---as indicated in Assumption {\bf A8}. \\

\vspace{.73cm}
Current research is under way along the following directions.
\begite
\item To replace  Assumption {\bf A8}, which is rather restrictive, by some integrability-like condition that is verified in some practical examples. 
\item Investigate alternatives to the classical LaSalle analysis used in the proof  of Proposition \ref{pro5} to relax the restrictive assumption of injectivity of $\nabla V_u(q_u)$, {\em e.g.}, invoking the Matrosov-like theorems of \cite{SCAetal}. Also, to identify classes of systems for which it is possible to prove boundedness of trajectories without Lyapunov stability. 
\item The theoretical analysis of the practical implementation of the PID using an approximate differentiator \eqref{appdif} seems feasible using singular perturbation arguments. However, as usual with this approach, the resulting results might be too conservative to be of practical interest.
\item It is necessary to get a better understanding---hopefully in some geometric or coordinate-free terms---of the class of systems verifying the key Assumptions {\bf A1-A4}.  
\item To further explore the relationship between total energy shaping PBC and the proposed PID the following two questions should be explored.
\begite
\item Compare their transient performances, for instance, investigating the flexibility to locate the eigenvalues of their tangent approximations---as done in \cite{KOT} for IDA-PBC. 
\item In \cite{DONORTROM} it has been shown that the PID controller of \cite{DONetal} can be recasted as a classical  IDA-PBC with {\em generalized dissipative forces}. What can be said about the nature of these forces? 
\endite
\endite
%
%

%
\section*{A. Proof of Proposition \ref{pro1}}

The first step in the proof is to invoke Assumptions {\bf A1-A4}  to rewrite the system \eqref{lagr} in a more suitable form. Towards this end, we recall the following well--known identity  given in eq. 3.19, page 72 of \cite{KELbook}  
\begequ
\lab{cdotq}
C(q,\dot{q}) \dot{q} = \left[ \nabla_q(M(q)\dot{q}) -\frac 12 \nabla_q^\top (M(q) \dot{q}) \right] \dot{q} ,
\endequ
Now, under  Assumptions {\bf A2} and {\bf A3},  the second function within brackets of the left hand side  of \eqref{cdotq}  takes the form
\begin{equation*}
\label{naM}
\nabla_q(M\dot{q}) = \left[ \begarr{cc} \nabla_{q_u}(m_{au}^\top \dot{q}_a + m_{uu} \dot{q}_u)   & 0_{s \times m} \\ \nabla_{q_u}( m_{au} \dot{q}_u) & 0_{m \times m}    \endarr \right],
\end{equation*}
where here, and throughout the rest of the proof, to simplify the notation, the arguments of some mappings will be omitted. Hence, using    \eqref{naM}, the vector  \eqref{cdotq} can be rewritten as follows 
\begin{eqnarray*}
C\dot q &=& \left[ \begarr{cc} \nabla_{q_u} (m_{au}^\top \dot{q}_a  + m_{uu} \dot{q}_u) - \frac 12 \nabla^\top_{q_u}(m_{au}^\top \dot{q}_a   +m_{uu} \dot{q}_u)  & - \frac 12 \nabla^\top_{q_u} (m_{au} \dot{q}_u)   \\ \nabla_{q_u}( m_{au} \dot{q}_u)  &  
0_{m \times s}   \endarr \right] \dot q \nonumber \\
 &=& \left[ \begarr{c}    \nabla_{q_u}(m_{uu} \dot{q}_u)\dot q_u- \frac 12 \nabla^\top_{q_u}(m_{uu} \dot{q}_u)\dot q_u  +\nabla_{q_u} (m_{au}^\top \dot{q}_a)\dot q_u - \frac 12 \nabla^\top_{q_u}(m_{au}^\top \dot{q}_a)\dot q_u     - \frac 12 \nabla^\top_{q_u} (m_{au} \dot{q}_u) \dot q_a  \\ \nabla_{q_u}( m_{au} \dot{q}_u) \dot q_u     \endarr \right] \nonumber \\
 &=&\left[ \begarr{c}  \nabla_{q_u}(m_{uu} \dot{q}_u) \dot q_u   - \frac 12 \nabla^\top_{q_u}(m_{uu} \dot{q}_u) \dot q_u  +\nabla_{q_u} (m_{au}^\top \dot{q}_a)\dot q_u - \nabla^\top_{q_u} (m_{au} \dot{q}_u) \dot q_a  \\ \nabla_{q_u}( m_{au} \dot{q}_u) \dot q_u     \endarr \right] \nonumber  \\
 &=& \left[ \begarr{c}  C_{mu}(q_u, \dot q _u) \dot q_u   + D_{mu}(q_u, \dot q)   \\ \nabla_{q_u}( m_{au} \dot{q}_u) \dot q_u   \endarr \right], 
\label{Corio}
\end{eqnarray*}
with 
\begin{eqnarray} 
C_{mu}(q_u, \dot q_u)&:=& \nabla_{q_u}(m_{uu} \dot{q}_u)    - \frac 12 \nabla^\top_{q_u}(m_{uu} \dot{q}_u)  \label{Cmuu}  \\
D_{mu}(q_u, \dot q)&:=& \nabla_{q_u} (m_{au}^\top \dot{q}_a)\dot q_u - \nabla^\top_{q_u} (m_{au} \dot{q}_u) \dot q_a, 
\label{Dmu} 
\end{eqnarray}
where we have used the fact that
\begequ
\lab{ide}
\nabla_{q_u}^\top (m_{au}^\top \dot{q}_a) \dot{q}_u = \nabla_{q_u}^\top (m_{au} \dot{q}_u) \dot{q}_a.
\endequ
Consequenty, invoking Assumption {\bf A4}, the system \eqref{lagr} with the inner--loop control \eqref{utau}  can be written as 
\begin{eqnarray}
m_{uu}\ddot q_u&=& -[m_{au}^\top \ddot q_a+C_{mu}\dot q_u +D_{mu}+\nabla V_u   ] \label{muuddqu} \\
 \ddot q_a &=&-m_{aa}^{-1}[m_{au} \ddot q_u+\nabla_{q_u} (m_{au} \dot q_u)\dot q_u - u]. \label{maaddqa}
\end{eqnarray}
Moreover,  replacing \eqref{maaddqa} in \eqref{muuddqu}, we get
\begin{eqnarray}
m_{uu}^s\ddot q_u&=&m_{au}^\top m_{aa}^{-1}[\nabla_{q_u} (m_{au} \dot q_u)\dot q_u- u]-[C_{mu}\dot q_u +D_{mu}+\nabla V_u   ]  \label{ddotqua}
\end{eqnarray}

We proceed now to prove \eqref{dotha}. The time derivative of the storage function \eqref{newstofun1} along the solution of \eqref{ddotqua} yields 
\begin{eqnarray*}
\quad \dot{ H}_u&=& \dot q_u^\top \Big[  m_{au}^\top m_{aa}^{-1}\Big(\nabla_{q_u} (m_{au} \dot q_u)\dot q_u -u \Big)-\Big(C_{mu}\dot q_u +D_{mu} \Big) 
    \Big] +\hal \dot q_u^\top \dot m_{uu}^s  \dot q_u   \nonumber \\
 &=&  u^\top  y_u + L(q,\dot q)
 \label{dotbHa}
\end{eqnarray*}
where we used \eqref{outlag} and defined the matrix  
\begin{equation}
\label{Sqdq}
L(q,\dot q):=\hal \dot q_u^\top \Big[   \dot m^s_{uu}  + 2 m_{au}^\top m_{aa}^{-1}\nabla_{q_u} (m_{au} \dot q_u) - 2 C_{mu} \Big] \dot q_u -\dot q_u^\top D_{mu}. 
\end{equation}
To complete the proof we will show that $L(q,\dot q)=0$. This is established using Assumption {\bf A3}, \eqref{ide}, the definition of $m_{uu}^s$ given in  \eqref{ma} and the following identities 
\begin{eqnarray*}
\nonumber
\dot q_u^\top D_{mu}& = & 0 \label{dqCdq} \\
\dot q_u^\top [ \dot m_{uu}  -2C_{mu} ] \dot q_u &=& 0 \nonumber \\
\dot q_u^\top \left[  \frac{d}{dt} (m_{au}^\top m_{aa}^{-1} m_{au})  - 2 m_{au}^\top m_{aa}^{-1}\nabla_{q_u} (m_{au} \dot q_u) \right] \dot q_u &=& 0. \nonumber
\end{eqnarray*}

To prove \eqref{dothu} we, first, notice that the  storage function \eqref{newstofun2} can be rewritten as follows 
 \begin{eqnarray}
 H_a &=& \hal \dot q^\top M_a \dot q  \pm H_u \nonumber \\
 &=& \hal \dot q^\top M \dot q -  H_u + V_u(q_u). 
  \label{Ha1}
 \end{eqnarray}
Hence, taking the time derivative of \eqref{Ha1} along the solutions of \eqref{lagr} with the inner--loop control \eqref{utau} yields 
\begin{eqnarray*}
\dot {\bar H}_a&=& \hal \dot q^\top \dot{M} \dot q + \dot q^\top  \left(  -C(q,\dot q) \dot q -\lef[{c} \nabla V_u \\ u  \rig] \right)   -\dot{ H}_u + \nabla V_u ^\top \; \dot q_u \nonumber \\ 
&=&\dot q_a^\top u - \dot{ H}_u \nonumber\\
&=& \dot q_a^\top u - y_u^\top u  \nonumber\\
&=& u^\top  y_a.
\end{eqnarray*}
\qed

%
  %
\section*{B. Proof of Proposition \ref{pro4}}
  
The proof follows very closely the proof of Proposition \ref{pro1} with the only difference of the  inclusion of $V_a(q_a)$ via Assumption {\bf A8}. 
Taking the time derivatives of the storage functions \eqref{newstofun1a} along the solutions of \eqref{ddotqua} yields

\begin{eqnarray*}
\quad \dot{ \bar H}_u&=& \dot q_u^\top \Big[  m_{au}^\top m_{aa}^{-1}[\nabla_{q_u} (m_{au} \dot q_u)\dot q_u+ \nabla V_a-\tau]-(C_{mu}\dot q_u +D_{mu} +\nabla V_u) 
    \Big] +\hal \dot q_u^\top \dot m_{uu}^s  \dot q_u   \nonumber \\
   && + \nabla V_u^\top \dot q_u - \dot V_0  \nonumber \\
 &=&  \tau^\top  y_u + L(q,\dot q) - \dot V_0   + \kappa^\top   y_u  \nonumber \\
 &=&  \tau^\top  y_u
 \label{dotbHa}
\end{eqnarray*}
where $L(q, \dot q)$ is given by \eqref{Sqdq}, the second identity is obtained recalling that $L(q,\dot q)=0$ and invoking  {\bf A8}, and the last one follows from \eqref{V0}. 

We now proceed to prove passivity of the operator $\tau \mapsto   y_a$. Similarly to \eqref{Ha1} the  storage function \eqref{newstofun2a} can be rewritten as 
 \begin{eqnarray*}
  \bar H_a &=& \hal \dot q M_a \dot q + V_a(q_a) + V_0(q_u) \pm \bar H_u 
   \nonumber \\
 &=& \hal \dot q M \dot q +V_a(q_u) +V_0(q_u) + \left( \hal \dot q_u m_{uu}^s \dot q_u + V_u(q_u) - V_0(q_u)\right) -\bar H_u \nonumber \\
 &=&  \hal \dot q M \dot q -   \bar H_u + V_a(q_a) + V_u(q_u), 
 \label{barHa1}
 \end{eqnarray*}
and computing its   time derivatives  along the solutions of \eqref{lagr}  we get 
\begin{eqnarray*}
\dot { \bar H}_a&=& \dot q^\top  \left(  -C(q,\dot q) \dot q -\lef[{c} \nabla V_u \\ \nabla V_a  \rig]  + \lef[{c} 0 \\ \tau  \rig] \right)  +\hal \dot q^\top \dot{M} \dot q -\dot{ \bar H}_u + \nabla^\top V_a \;  \dot q_a + \nabla V_u ^\top \; \dot q_u \nonumber \\ 
&=&\dot q_a^\top \tau - \dot{ \bar H}_u   - \dot q_u^\top \nabla V_u + \nabla^\top V_u \; \dot q_u - \dot q_a^\top  \nabla V_a  + \nabla^\top V_a \; \dot q_a \nonumber\\
&=& \dot q_a^\top \tau + \dot q_u^\top m_{au}^\top m_{aa}^{-1} \tau \nonumber\\
&=& \tau^\top  y_a.
\end{eqnarray*}
\qed

 %
\section*{C. Proof of Proposition \ref{pro5}}
%
It has been shown in Subsection \ref{subsec42} that, independently of Assumption {\bf A7}, the function $H_d(q,\dot q)$ given in \eqref{tilhd} verifies
$$
\dot H_d =- \|  y_d \|^2_{K_P},
$$
where $y_d$ is defined in \eqref{yd1}. Invoking La Salle's invariance principle \cite{KHA}, we can conclude that all bounded trajectories converge to the maximum invariant set contained in 
$$
\mathcal X:=\{ (q,\dot q) \in \rea^n \times \rea^n\;|\;  k_a \dot q_a + (k_a-k_u) m_{aa}^{-1} m_{au}(q_u) \dot q_u \equiv 0 \}.
$$
Now, from the PID controller \eqref{pidcon1} it is clear that $y_d \equiv 0$ implies $z_1=c_0$ and, consequently, $u=c_0$ where here, and throughout the rest of the proof, $c_0$ denotes a (generic) constant vector (of suitable dimensions).  

We now compute the time derivative of $ y_d$, which results as follows
\begequarrs
\dot{ y}_d &=& (k_a-k_u) m_{aa}^{-1} \nabla_{q_u}(m_{au} \dot q_u) \dot q_u + (k_a-k_u) m_{aa}^{-1} m_{au} \ddot q_u + k_a \ddot q_a 
\label{dotyd} \\
&=&  (k_a-k_u)  m_{aa}^{-1} u + k_u \ddot q_a, 
\endequarrs
where we have used \eqref{maaddqa} to get the second identity. From the equation above $\dot y_d=0$ and $u=c_0$ we conclude that, in $\mx$,  $\ddot{q}_a = c_0$ also. Since we are looking only at {\em bounded} trajectories this implies that $\ddot q_a=\dot q_a=0$ and $q_a=c_0$. Setting $\dot q_a=0$ and $y_d=0$ in \eqref{dotz1} yields $\dot{ V}_N(q_u)=0$ that, replaced in \eqref{dotvn} and invoking the {\em full rank} Assumption {\bf A9}, allows us to conclude that  $\dot q_u=0$, which implies that $q_u=c_0$. This proves that the only bounded trajectories living in the residual set $\mx$ are of the form $(q(t),\dot q(t))=(\bar q,0)$, with $\bar q$ a constant vector.

It only remains to prove that the only point $(\bar q,0)$, which is invariant to the dynamics, is with $\bar q=q^\star$. This follows from injectivity of $\nabla V_u(q_u)$ that, setting it equal to zero, ensures $\bar q_u=q_u^\star$ and the following chain of implications
\begequarrs
(q(t),\dot q(t))=(\bar q,0) & \Rightarrow & u=0 \\
                & \Rightarrow & \bar q_a=q_a^\star \\
                & \Rightarrow & \bar q = q^\star,
\endequarrs
where we have used  \eqref{maaddqa} in the first implication and  \eqref{uatequ} together with $\bar q_u=q_u^\star$ in the second one.

\qed

\end{document}